\documentclass[11pt]{amsart}
\usepackage{amsxtra}
\usepackage[all]{xy}
\usepackage{amssymb}
\usepackage{amsmath}
\usepackage{tikz-cd}
\usepackage{array}
\usepackage{booktabs}
\usepackage{ifpdf}
\usepackage{enumerate}

\ifpdf
\usepackage[colorlinks,final,backref=page,hyperindex]{hyperref}
\else
\usepackage[colorlinks,final,backref=page,hyperindex,hypertex]{hyperref}
\fi

\usepackage{tikz}
\usepackage[active]{srcltx}
\usepackage{cases}
\usepackage{appendix}

\theoremstyle{plain}

\newcommand{\cleqn}{\setcounter{equation}{0}}
\newcommand{\clth}{\setcounter{theorem}{0}}
\newcommand {\sectionnew}[1]{\section{#1}\cleqn\clth}

\newtheorem{theorem}{Theorem}[section]
\newtheorem{lemma}[theorem]{Lemma}
\newtheorem{definition-theorem}[theorem]{Definition-Theorem}
\newtheorem{proposition}[theorem]{Proposition}
\newtheorem{corollary}[theorem]{Corollary}
\newtheorem{definition}[theorem]{Definition}
\newtheorem{example}[theorem]{Example}
\newtheorem{remark}[theorem]{Remark}
\newtheorem{notation}[theorem]{Notation}
\newtheorem{assumption}[theorem]{Assumption}
\newtheorem{lemma-definition}[theorem]{Lemma-Definition}
\newtheorem{lemma-notation}[theorem]{Lemma-Notation}
\newtheorem{question}[theorem]{Question}
\newtheorem{remark-definition}[theorem]{Remark-Definition}
\newcommand \bth[1] { \begin{theorem}\label{t#1} }
\newcommand \ble[1] { \begin{lemma}\label{l#1} }

\newcommand \bpr[1] { \begin{proposition}\label{p#1} }
\newcommand \bco[1] { \begin{corollary}\label{c#1} }
\newcommand \bde[1] { \begin{definition}\label{d#1}\rm }
\newcommand \bex[1] { \begin{example}\label{e#1}\rm }
\newcommand \bre[1] { \begin{remark}\label{r#1}\rm }

\newcommand \bnota[1] {\begin{notation}\label{n#1}\rm }
\newcommand \bas[1] { \begin{assumption}\label{a#1}\rm }

\newcommand \bqu[1] { \begin{question}\label{q#1}\rm }

\newcommand {\ele} { \end{lemma} }

\newcommand {\epr} { \end{proposition} }
\newcommand {\eco} { \end{corollary} }
\newcommand {\ede} { \end{definition} }
\newcommand {\eex} { \end{example} }
\newcommand {\ere} { \end{remark} }
\newcommand {\enota} { \end{notation} }
\newcommand {\eas} {\end{assumption}}

\newcommand {\equ} {\end{question}}

\newcommand \lb[1]{\label{#1}}

\newcommand{\beqa}{\begin{eqnarray*}}                     
\newcommand{\eeqa}{\end{eqnarray*}}

\def \sC {{\scriptscriptstyle C}}

\def \wF_mn {\wF_m \times \wF_n}
\def \wF_mnC {\wF_{m, n, \, \sC}}

\def \wF {\widetilde{F}}

\usepackage{lipsum}

\newcommand\blfootnote[1]{%
  \begingroup
  \renewcommand\thefootnote{}\footnote{#1}%
  \addtocounter{footnote}{-1}%
  \endgroup
}
\begin{document}

\setlength{\baselineskip}{1.2\baselineskip}
\title[Manin triples, bialgebras and Yang-Baxter equation of $A_3$-associative algebras]
{Manin triples, bialgebras and Yang-Baxter equation of $A_3$-associative algebras}
\author{Yaxi Jiang}
\address{
School of Mathematical Sciences    \\
Zhejiang Normal University\\
Jinhua 321004              \\
China}
\email{jiangyx@zjnu.edu.cn}

\author{Chuangchuang Kang}
\address{
School of Mathematical Sciences    \\
Zhejiang Normal University\\
Jinhua 321004              \\
China}
\email{kangcc@zjnu.edu.cn}

\author{Jiafeng L\"u}
\address{
School of Mathematical Sciences    \\
Zhejiang Normal University\\
Jinhua 321004              \\
China}
\email{jiafenglv@zjnu.edu.cn}

\blfootnote{*Corresponding Author: Chuangchuang Kang. Email: kangcc@zjnu.edu.cn.}

\date{}
\begin{abstract}
$A_3$-associative algebra is a generalization of associative algebra and is one of the four remarkable types of Lie-admissible algebras, along with  associative algebra, left-symmetric algebra and right-symmetric algebra. This paper develops bialgebra theory for $A_3$-associative algebras. We introduce Manin triples and bialgebras for $A_3$-associative algebras, prove their equivalence using matched pairs of $A_3$-associative algebras, and define the $A_3$-associative Yang-Baxter equation and triangular $A_3$-associative bialgebras. Additionally, we introduce relative Rota-Baxter operators to provide skew-symmetric solutions of the $A_3$-associative Yang-Baxter equation.
\end{abstract}


\maketitle
\tableofcontents
\allowdisplaybreaks
\sectionnew{Introduction }\lb{intro}
\subsection{$A_3$-associative algebras and Lie-admissible algebras}
The notion of an $A_3$-associative algebra was  first given by Goze and Remm in \cite{Goze1}. It generalizes associative algebra and is one of the four remarkable types of Lie-admissible algebras, along with associative algebra, left-symmetric algebra and right-symmetric algebra \cite{Goze1}. Lie-admissible algebras play a significant role in exploring the symmetries of certain particles and have various applications in physics and mathematics, we refer to \cite{Albert,Ammar,Bai-Chen,Benkart2,Benkart1,ChenY,Jeong,Makhlouf} for more interesting research on Lie-admissible algebras.

An $A_3$-associative algebra is defined as an algebra $(A,\cdot)$ that satisfies the $A_3$-associative law:
\begin{equation*} 
  (x \cdot y)\cdot z +(y \cdot z)\cdot x +(z \cdot x)\cdot y = x\cdot (y \cdot z) +y\cdot (z \cdot x) +z\cdot (x \cdot y), \quad \forall~ x,y,z\in A.
\end{equation*}
This law can also be expressed using associators as:
 $$
 \sum_{\sigma\in A_3}(a_{\sigma(1)},a_{\sigma(2)},a_{\sigma(3)})=0,
 $$
where $a_1, a_2, a_3 \in A$, and $A_3$ denotes the alternating group of degree $3$. 
The $A_3$-associative law is called the $G_5$-associative law in \cite{Goze1}. 
The symmetric form above provides the basis for the terminology ``$A_3$'' and  has been instrumental in constructing  $A_3$-dendriform algebra \cite{Ospel}.
Associative algebras are a special case of $A_3$-associative algebras. By defining a skew-symmetric bilinear map $[\cdot,\cdot]: A \times A \rightarrow A$ as $[x,y] = x \cdot y - y \cdot x$, where $\cdot$ satisfies the $A_3$-associative law, the structure $(A, [\cdot,\cdot])$ becomes a Lie algebra. 
Furthermore, admissible Poisson algebras,  which belong to the class of $A_3$-associative algebras \cite{Goze2}, have been studied by many scholars from different aspects \cite{Liang,Remm1}. 

\subsection{The bialgebra theory for Lie-admissible algebras}
The Lie bialgebra consists of a Lie algebra and a Lie coalgebra satisfying the 1-cocycle condition, which plays an important role in the field of mathematical physics \cite{Chari,Drinfeld2}. For example, Lie bialgebras provide a valuable method for constructing solutions of the classical Yang-Baxter equation \cite{Drinfeld1}. In recent years, many scholars have also explored other types of bialgebra structures, such as Novikov-Poisson bialgebras \cite{LiB}, Rota-Baxter bialgebras \cite{BaiC5}, anti-pre Lie bialgebras \cite{LiuG1}, Mock-Lie bialgebras \cite{Benali}, Leibniz bialgebras \cite{Rezaei,Tang} and Poisson bialgebras \cite{Ni}.

Regarding the bialgebra theory for Lie-admissible algebras, bialgebra theories have been developed for specific examples of Lie-admissible algebras, including left-symmetric (or right-symmetric) bialgebras \cite{BaiC6} and associative bialgebras \cite{Joni}. $A_3$-associative algebras are a special class of Lie-admissible algebras,
which raises the natural question: do $A_3$-associative algebras possess corresponding bialgebras?

By drawing an analogy with Lie bialgebras, the author studied the bialgebras of left-symmetric algebras, which are equivalent to the matched pairs and symplectic doubles \cite{BaiC6}. Building on this foundational work, further studies have emerged, such as Hom-left-symmetric bialgebras \cite{SunQ}, left-symmetric conformal bialgebras \cite{HongY1} and Hom-left-symmetric conformal bialgebras \cite{GuoS}.

In \cite{Joni}, the authors introduced the concept of bialgebras of associative algebras, also referred to as infinitesimal bialgebras. These structures consist of an algebra and a coalgebra where the comultiplication satisfies the derivation property. Based on this, the authors studied infinitesimal Hopf algebras \cite{Aguiar1}, infinitesimal Hom-bialgebras \cite{Liul1} and infinitesimal BiHom-bialgebras \cite{Liul2}. In \cite{Aguiar3}, the author established the connection between infinitesimal bialgebras and Lie bialgebras, and introduced balanced infinitesimal bialgebras, which are both commutative and cocommutative. In the opposite algebra sense of balanced infinitesimal bialgebras, Bai studied antisymmetric infinitesimal bialgebras, which are equivalent to the double constructions of Frobenius algebras and matched pairs \cite{BaiC2}.

\subsection{The classical Yang-Baxter equation and relative Rota-Baxter operators}
The skew-symmetric solutions of the classical Yang-Baxter equation yield a class of Lie bialgebras \cite{Drinfeld1}, with applications in inverse scattering theory \cite{Faddeev}, classical integrable systems \cite{BaiC8} and quantum groups \cite{Drinfeld2,Chari}. 
The operator form of this equation corresponds to a relative Rota-Baxter operator of weight 0 (also known as an $\mathcal{O}$-operator \cite{BaiC7}). Relative Rota-Baxter operators generalize Rota-Baxter operators and have diverse applications, including Hom-Lie triple systems \cite{LiY}, Lie-Yamaguti algebras \cite{ShengY2} and relative Poisson algebras \cite{LiuG2}. This paper aims to study the $A_3$-associative analogues of these constructions, focuses on the $A_3$-associative Yang-Baxter equation and its corresponding relative Rota-Baxter operators.

Since $A_3$-associative algebra generalizes associative algebra, its bialgebra structure should generalize the bialgebra of associative algebra.
For infinitesimal bialgebras, Aguiar introduced the associative Yang-Baxter equation and proved that their solutions yield quasitriangular infinitesimal bialgebras \cite{Aguiar1}.
Aguiar also showed that a solution of the associative Yang-Baxter equation corresponds to a relative Rota-Baxter operator of weight $0$ \cite{Aguiar2}. 
Additionally, the skew-symmetric solution of the associative Yang-Baxter equation corresponds to the skew-symmetric part of the relative Rota-Baxter operator, which provides the double construction of Frobenius algebras \cite{BaiC2}.

\subsection{Outline of the paper}
In this paper, we introduce the notions of Manin triples and bialgebras for $A_3$-associative algebras. We prove the equivalence between Manin triples and bialgebras using matched pairs of $A_3$-associative algebras. Additionally, we define the $A_3$-associative Yang-Baxter equation and triangular $A_3$-associative bialgebras, and introduce Rota-Baxter operators to provide skew-symmetric solutions of the $A_3$-associative Yang-Baxter equation.
The paper is organized as follows. 

In Section 2, we present essential preliminaries. Specifically, we give the representation of $A_3$-associative algebra. However, the triple $(r^*,l^*,A)$ is also a representation, an additional condition is required. In such cases, we refer to this type of representation as admissible. Then, we define an admissible $A_3$-associative algebra. Furthermore, we introduce the concepts of matched pairs of $A_3$-associative algebras and quadratic $A_3$-associative algebra. 

In Section 3, we introduce the concept of a Manin triple of $A_3$-associative algebras, which defines a quadratic $A_3$-associative algebra $(A \oplus A',\cdot_\mathfrak{d},\mathcal{B}_\mathfrak{d})$. We also introduce the notion of an $A_3$-associative bialgebra, which is showed to be equivalent to a Manin triple of $A_3$-associative algebras via a specific matched pair of $A_3$-associative algebras.

In Section 4, we study triangular $A_3$-associative bialgebras constructed from a skew-symmetric solution of the $A_3$-associative Yang-Baxter equation. We introduce relative Rota-Baxter operators to describe the operator forms of these solutions.

Throughout this paper, unless otherwise specified, all the vector spaces and algebras are finite-dimensional over an algebraically closed field $\mathbb{K}$ of characteristic zero, $\mathbb{C}$ is the complex field.
Let $V$ be a vector space, $V^*$ be the dual space of $V$. For each positive integer $k$, we identify the tensor product $\otimes^k V$ with the space of multilinear maps from $\underbrace{V^*\times \cdots \times V^*}_{k-times}$ to $\mathbb{K}$, such that
\begin{equation}\label{eq:action}
\langle\xi_1\otimes \cdots\otimes \xi_k,v_1\otimes\cdots\otimes v_k\rangle=\langle\xi_1,v_1\rangle\cdots\langle\xi_k,v_k\rangle,\quad \forall ~\xi_1,\cdots,\xi_k\in V^*,v_1,\cdots,v_k\in V,
\end{equation}
where $\langle \xi_i,v_i\rangle=\xi_i(v_i)$, $1\leq i\leq k$. For $v_1,\cdots,v_k\in V$, define that
$$
v_1\wedge v_2 \wedge \cdots\wedge v_k=\sum_{\sigma\in S_k}sgn(\sigma)v_{\sigma(1)}\otimes v_{\sigma(2)}\cdots v_{\sigma(k)}\in \wedge^kV\subset \otimes^kV.
$$

\section{$A_3$-associative algebras and their representations}

In this section, we introduce the notions of representations and matched pairs of $A_3$-associative algebras. We also give the concept of quadratic $A_3$-associative algebra.

\subsection{Representations and matched pairs of $A_3$-associative algebras}

\begin{definition}{\rm(\cite{Ospel})}
An \textbf{$A_3$-associative algebra} $(A,\cdot)$ is a vector space $A$ with a bilinear map $\cdot: A \times A \rightarrow A$ satisfying
\begin{equation} \label{eq:asso-alg}
  (x \cdot y)\cdot z +(y \cdot z)\cdot x +(z \cdot x)\cdot y = x\cdot (y \cdot z) +y\cdot (z \cdot x) +z\cdot (x \cdot y), \quad \forall~ x,y,z\in A.
\end{equation}
\end{definition}

\begin{definition}
Let $(A_1,\cdot_1)$ and $(A_2,\cdot_2)$ be two $A_3$-associative algebras. A \textbf{homomorphism} of $A_3$-associative algebra is a linear map $\varphi:A_1 \rightarrow A_2$ such that
\begin{equation} \label{eq:iso}
  \varphi(x \cdot_1 y) =\varphi(x) \cdot_2 \varphi(y), \quad \forall~ x,y\in A_1.
\end{equation}
\end{definition}

\begin{definition} 
Let $(A,\cdot)$ be an $A_3$-associative algebra, $A_1$ be a subspace of $A$. If for all $x,y\in A_1$, $x\cdot y \in A_1$, then $(A_1,\cdot)$ is called a  \textbf{subalgebra} of $(A,\cdot)$.
\end{definition}

\begin{definition} {\rm(\cite{Gerstenhaber})}
An \textbf{associative algebra} is a pair $(A,\cdot)$, where $A$ is a vector space and $\cdot:A \times A \rightarrow A$ is a bilinear map satisfying
\begin{equation} \label{eq:associa-alg}
  (x\cdot y)\cdot z=x\cdot(y\cdot z),  \quad \forall~ x,y,z\in A. 
\end{equation}
\end{definition}

\begin{example}
An associative algebra is an $A_3$-associative algebra. By rotating \(x, y, z\) in \eqref{eq:associa-alg}, we obtain three new equations. Adding these together yields \eqref{eq:asso-alg}.
\end{example}

\begin{definition} {\rm(\cite{Goze2})}
An \textbf{admissible Poisson algebra} is a pair $(A,\cdot)$, where $A$ is a vector space and $\cdot:A \times A \rightarrow A$ is a bilinear map satisfying
\begin{equation} \label{eq:ad-PA}
  3(x\cdot y)\cdot z = 3x\cdot(y\cdot z) + (x\cdot z)\cdot y + (y\cdot z)\cdot x - (y\cdot x)\cdot z - (z\cdot x)\cdot y, \quad \forall~ x,y,z\in A.
\end{equation} 
\end{definition}

\begin{example} {\rm(\cite{Goze2})}
An admissible Poisson algebra $(A,\cdot)$ is an $A_3$-associative algebra. By rotating \(x, y, z\) in \eqref{eq:ad-PA}, we obtain three new equations. Adding these together yields \eqref{eq:asso-alg}.
\end{example}

For an algebra $(A,\cdot)$, the algebra $A^-$ is defined as a pair $(A,[\cdot,\cdot])$, where $[x,y]=x\cdot y -y\cdot x$ for all $x,y \in A$, the algebra $(A,\cdot)$ is called admissible if $A^-$ is a Lie algebra.

\begin{definition} {\rm(\cite{Albert})}
A \textbf{Lie-admissible algebra} is an algebra $(A,\cdot)$ satisfying
\begin{equation} \label{eq:Lie-ad-jaco}
    (xy-yx)z + (yz-zy)x + (zx-xz)y = x(yz-zy) + y(zx-xz) + z(xy-yx), \quad \forall~ x,y,z\in A,
\end{equation}
where $\cdot$ is omitted.
\end{definition}

In fact, if define an anticommutative bilinear map $[\cdot,\cdot]:A \times A \rightarrow A$ as 
\begin{equation} \label{eq:Lie-ad}
  [x,y]=x\cdot y-y\cdot x, \quad \forall~ x,y\in A.
\end{equation}
Then \eqref{eq:Lie-ad-jaco} is the Jacobi identity, that is,
\begin{equation*}
  [[x,y],z]+[[y,z],x]+[[z,x],y]=0, \quad \forall~ x,y,z\in A.
\end{equation*}
If we replace $x, y, z$ with $y, x, z$ in \eqref{eq:asso-alg}, subtract it from \eqref{eq:asso-alg}, we obtain \eqref{eq:Lie-ad-jaco}. Hence, we have the following conclusion.

\begin{proposition} {\rm(\cite{Goze1})}
An $A_3$-associative algebra is a Lie-admissible algebra.
\end{proposition}

\begin{definition}
A \textbf{representation} of an $A_3$-associative algebra $(A,\cdot)$ is a triple $(l,r,V)$, where $V$ is a vector space and $l,r:A \rightarrow End(V)$ are linear maps such that 
\begin{equation} \label{eq:repre}
  l(x \cdot y)v - r(x \cdot y)v + r(x)l(y)v -l(x)l(y)v + r(y)r(x)v -l(y)r(x)v =0,\quad \forall~ x,y\in A,~v \in V.
\end{equation}
\end{definition}

\begin{definition} {\rm(\cite{BaiC2})} 
A representation of an associative algebra $(A,\cdot)$ is a triple $(l,r,V)$, where $V$ is a vector space and $l,r:A \rightarrow End(V)$ are linear maps such that 
\begin{equation} \label{eq:asso-re}
  l(x\cdot y)v=l(x)l(y)v,~r(x\cdot y)v=r(y)r(x)v,~l(x)r(y)v=r(y)l(x)v,\quad \forall~ x,y \in A,~v \in V.
\end{equation}
\end{definition}

\begin{proposition}
Let $(l,r,V)$ be a representation of associative algebra. Then $(l,r,V)$ is a representation of an $A_3$-associative algebra.
\end{proposition}
\begin{proof}
Let $(A,\cdot)$ be an associative algebra.
For all $x,y \in A,~v \in V$, we have
\begin{eqnarray*}
  && l(x \cdot y)v - r(x \cdot y)v + r(x)l(y)v -l(x)l(y)v + r(y)r(x)v -l(y)r(x)v \\
  &\overset{\eqref{eq:asso-re}}{=}& l(x)l(y)v -r(y)r(x)v + r(x)l(y)v -l(x)l(y)v + r(y)r(x)v -r(x)l(y)v \\
  &=& 0, 
\end{eqnarray*}
where the above notation $\overset{\eqref{eq:asso-re}}{=}$ indicates
that the equation \eqref{eq:asso-re} is used to establish the equality. Hence, the conclusion holds.
\end{proof}

\begin{definition}
An \textbf{adjoint representation} of an $A_3$-associative algebra $(A,\cdot)$ is a triple $(L,R,A)$, where $L,R:A \rightarrow End(A)$ are two linear maps defined by
\begin{equation*}
  L(x)y =x \cdot y = R(y)x, \quad \forall~ x,y\in A.
\end{equation*}
\end{definition}

\begin{remark} {\rm (\cite{BaiC2})}
Let $(A,\cdot)$ be an associative algebra, $L:A \rightarrow End(A)$ with $x \mapsto L(x)$ and $R:A \rightarrow End(A)$ with $x \mapsto R(x)$ be two linear maps satisfying $L(x)(y)=xy$, $R(x)(y)=yx$, for any $x,y \in A$. Then $(L,R,A)$ is the adjoint representation of associative algebra $(A,\cdot)$, 
which is also the adjoint representation of $A_3$-associative algebra. 
\end{remark}

The dual action of associative algebras was given in \cite{BaiC2}. Let $(l,r,V)$ be a representation of an associative algebra $(A,\cdot)$, $l^*,r^*:A \rightarrow End(V^*)$ be two linear maps. The dual actions of left and right multiplication are defined as follows:
\begin{equation} \label{eq:dual}
   \langle l^*(x)u^*,v \rangle = \langle u^*,l(x)v \rangle, \langle r^*(x)u^*,v \rangle = \langle u^*,r(x)v \rangle, \quad \forall~ x\in A,~u^* \in V^*,~v \in V.
\end{equation}

\begin{remark}
In \cite{Ni} the dual action of the representation $(l,r,V)$ of Poisson algebras is defined by
\begin{equation}\label{eq:n-action}
  \langle l^*(x)u^*,v \rangle= - \langle u^*,l(x)v \rangle, ~ \langle r^*(x)u^*,v \rangle= - \langle u^*,r(x)v \rangle, \quad \forall~ x\in A,~u^* \in V^*,~v \in V,
\end{equation}
where $l^*,r^*:A \rightarrow End(V^*)$ are the linear maps. If the left multiplication $``l"$ and right multiplication $``r"$ of the above algebras satisfy
\begin{equation*}
  [r(y),l(x)] - [r(x),l(y)]= 2r(y)r(x) - 2l(x)l(y),\quad \forall~ x,y\in A,
\end{equation*}
then $(r^*,l^*,V^*)$ is a representation of $A_3$-associative algebra $(A,\cdot)$. This dual action \eqref{eq:n-action} also appears in the dual representations of Lie algebras \cite{Drinfeld1}, left-symmetric algebras \cite{BaiC6}, Novikov-Poisson algebras \cite{LiB}, anti-pre Lie algebras \cite{LiuG1}, Leibniz algebras \cite{Tang}, and left-Alia algebras \cite{KangC}. The dual action of an associative algebra is defined by \eqref{eq:dual}. In this paper, we continue to use that definition without adding a negative sign on the right-hand side.
\end{remark}

 To generalize the dual representation of associative algebras, we adopt \eqref{eq:dual} as the definition of the dual action of $A_3$-associative algebra. 
In \cite[Lemma 2.1.2]{BaiC2}, $(r^*,l^*,V^*)$ is a representation of associative algebra. Next, we consider when $(r^*,l^*,V^*)$ is a representation of $A_3$-associative algebra.

\begin{proposition} \label{co-rep}
Let $(A,\cdot)$ be an $A_3$-associative algebra and $(l,r,V)$ be a representation. Then $(r^*,l^*,V^*)$ is a representation if and only if
\begin{equation} \label{eq:ad-repre}
  r(y)l(x)v - l(x)r(y)v + r(x)l(y)v - l(y)r(x)v =0, \quad \forall~ x,y \in A,~v \in V.
\end{equation}
\end{proposition}
\begin{proof}
For all $x,y \in A,~u^* \in V^*,~v \in V$, we have
\begin{eqnarray*}
  && \langle \big(r^*(x \cdot y) - l^*(x \cdot y) + l^*(x)r^*(y) - r^*(x)r^*(y) + l^*(y)l^*(x) - r^*(y)l^*(x) \big)u^*,v \rangle \\
  &\overset{\eqref{eq:dual}}{=}& \langle u^*, \big(r(x \cdot y) - l(x \cdot y) + r(y)l(x) - r(y)r(x) + l(x)l(y) - l(x)r(y)\big)v \rangle \\
  &\overset{\eqref{eq:ad-repre}}{=}& \langle u^*, \big(r(x \cdot y) - l(x \cdot y) + l(y)r(x) - r(y)r(x) + l(x)l(y) - r(x)l(y)\big)v \rangle  \\
  &\overset{\eqref{eq:repre}}{=}& 0.
\end{eqnarray*}
Therefore, $(r^*,l^*,V^*)$ is a representation of an $A_3$-associative algebra $(A,\cdot)$ if and only if \eqref{eq:ad-repre} holds.
\end{proof}

\begin{remark}
In \cite[Proposition 3.3]{BaiC6}, $(r^*-l^*,l^*,V^*)$ is a representation of left-symmetric algebras. 
If $(A,\cdot)$ is an $A_3$-associative algebra with a representation $(l,r,V)$, then $(r^*-l^*,l^*,V^*)$ is a representation also requiring extra condition.
In fact, for all $x,y \in A,~u^* \in V^*,~v \in V$, 
\begin{eqnarray*}
   \langle \big((r^*-l^*)(x \cdot y) - l^*(x \cdot y) + l^*(x)(r^*-l^*)(y) - (r^*-l^*)(x)(r^*-l^*)(y)\\ + l^*(y)l^*(x) 
   - (r^*-l^*)(y)l^*(x) \big)u^*,v \rangle\neq0.
\end{eqnarray*}

\end{remark}

\begin{definition} \label{admissible-re}
Let $(A,\cdot)$ be an $A_3$-associative algebra. The representation $(l,r,V)$ of $(A,\cdot)$ is called \textbf{admissible} if $(l,r,V)$ satisfies  \eqref{eq:ad-repre}.
\end{definition}

\begin{proposition} \label{dual-ad-repre}
Let $(l,r,V)$ be an admissible representation of $A_3$-associative algebra $(A,\cdot)$. If $l^*,r^*:A \rightarrow End(V^*)$ are two linear maps defined by \eqref{eq:dual}, then $(r^*,l^*,V^*)$ is an admissible representation.
\end{proposition}
\begin{proof}
By \eqref{eq:dual} and \eqref{eq:ad-repre}, we have
\begin{equation} \label{eq:ad-ad-repre}
  l^*(x)r^*(y)u^* -r^*(y)l^*(x)u^* +l^*(y)r^*(x)u^* -r^*(x)l^*(y)u^* =0, \quad \forall~ x,y \in A,~ u^* \in V^*.
\end{equation}
Then by Proposition \ref{co-rep} and Definition \ref{admissible-re}, we have $(r^*,l^*,V^*)$ is an admissible representation.
\end{proof}

\begin{definition}
The $A_3$-associative algebra $(A,\cdot)$ is called \textbf{admissible} if the multiplication $\cdot:A \times A \rightarrow A$ satisfies
\begin{equation} \label{eq:admissible}
  (x\cdot z)\cdot y - x\cdot(z\cdot y) = y\cdot(z\cdot x) - (y\cdot z)\cdot x, \quad \forall~ x,y,z\in A.
\end{equation}
\end{definition}

\begin{remark}
Note that \eqref{eq:admissible} is an anti-flexible identity in \cite{Mafoya}.
\end{remark}

\begin{definition}
Let $(A,\cdot)$ be an admissible $A_3$-associative algebra with an adjoint representation $(L,R,A)$. 
Then $(R^*,L^*,A^*)$ is a representation  of $(A,\cdot)$, which is called \textbf{coadjoint representation}. 
\end{definition}

\begin{example} \label{ad-asso-example}
Let $A$ be a $2$-dimensional vector space over the complex field $\mathbb{C}$ with a basis $\{e_1,e_2\}$. The multiplications on $A$ are given by
\begin{equation*}
  e_1 \cdot e_1 = e_1 \cdot e_2 = e_2 \cdot e_1 = e_1 + 2e_2,~e_2 \cdot e_2 =e_2.
\end{equation*}
Then $(A,\cdot)$ is a $2$-dimensional admissible $A_3$-associative algebra. But $(A,\cdot)$ is not an associative algebra, because
\begin{eqnarray*}
  && (e_1 \cdot e_2) \cdot e_2 - e_1 \cdot (e_2 \cdot e_2) 
  = (e_1 + 2e_2) \cdot e_2 - e_1 \cdot e_2 \\
  &=& (e_1 + 2e_2) + 2e_2 - (e_1 + 2e_2) 
  \neq 0.
\end{eqnarray*} 
\end{example}

\begin{definition} {\rm(\cite{Vinberg})}
A \textbf{left-symmetric algebra} is a pair $(A,\cdot)$, where $A$ is a vector space and $\cdot:A\times A\rightarrow A$ is a bilinear map on $A$ satisfying
\begin{equation} \label{eq:pre-Lie}
  (x\cdot y)\cdot z - x\cdot(y\cdot z)=(y\cdot x)\cdot z - y\cdot(x\cdot z), \quad \forall~ x,y,z \in A.
\end{equation}
\end{definition}

\begin{proposition}
If the algebra $(A,\cdot)$ is both admissible $A_3$-associative algebra and left-symmetric algebra, then $(A,\cdot)$ is an associative algebra.
\end{proposition}
\begin{proof}
Replace $x,y,z$ in \eqref{eq:asso-alg} with $y,x,z$, then subtract \eqref{eq:pre-Lie} and add \eqref{eq:admissible}, we have
\begin{equation*}
  (x\cdot z) \cdot y = x\cdot (z\cdot y), \quad \forall~ x,y,z \in A.
\end{equation*}
Therefore, the conclusion holds.
\end{proof}

\begin{definition} {\rm(\cite{Vinberg,Koszul})}
A \textbf{right-symmetric algebra} is a pair $(A,\cdot)$, where $A$ is a vector space and $\cdot:A\times A\rightarrow A$ is a bilinear map on $A$ satisfying
\begin{equation} \label{eq:right-s}
  (x\cdot y)\cdot z - x\cdot(y\cdot z)=(x\cdot z)\cdot y - x\cdot(z\cdot y), \quad \forall~ x,y,z \in A.
\end{equation}
\end{definition}

\begin{proposition}
If the algebra $(A,\cdot)$ is both admissible $A_3$-associative algebra and right-symmetric algebra, then $(A,\cdot)$ is an associative algebra.
\end{proposition}
\begin{proof}
Replace $x,y,z$ in \eqref{eq:asso-alg} with $y,x,z$, then subtract \eqref{eq:right-s} and add \eqref{eq:admissible}, we have
\begin{equation*}
    (y\cdot x) \cdot z = y\cdot (x\cdot z), \quad \forall~ x,y,z \in A.
\end{equation*}
Therefore, the conclusion holds.
\end{proof}

\begin{definition} \label{re-equivalent}
Let $(l,r,V)$ and $(l',r',V')$ be representations of an $A_3$-associative algebra $(A,\cdot)$. These two representations $(l,r,V)$ and $(l',r',V')$ are \textbf{equivalent} if there is a linear isomorphism $\phi:V\rightarrow V'$ such that
\begin{equation} \label{eq:re-iso}
  \phi \big(l(x)v \big) = l'(x)\phi(v),~\phi \big(r(x)v \big)=r'(x)\phi(v), \quad \forall~ x\in A,~v \in V.
\end{equation}
\end{definition}

\begin{definition}
Let $(A,\cdot)$ and $(B,\circ)$ be $A_3$-associative algebras and $l_A,r_A: A\rightarrow End(B)$, $l_B,r_B: B\rightarrow End(A)$ be linear maps.  Then $\big((A,\cdot),(B,\circ),l_A,r_A,l_B,r_B \big)$ is a \textbf{matched pair of $A_3$-associative algebras} if $(A \oplus B,\ast)$ is an $A_3$-associative algebra, where for all $x,y\in A,~a,b \in B$, $\ast:(A \oplus B)\otimes (A \oplus B) \rightarrow A \oplus B$ is a bilinear map satisfying
\begin{equation}\label{eq:mp-ds}
  (x+a)\ast (y+b) = x \cdot y +l_B(a)y +r_B(b)x + a \circ b +l_A(x)b +r_A(y)a.
\end{equation}
\end{definition}

\begin{proposition} \label{matched pair}
Let $(A,\cdot)$ and $(B,\circ)$ be $A_3$-associative algebras and $l_A,r_A: A\rightarrow End(B)$, $l_B,r_B: B\rightarrow End(A)$ be linear maps. Then $\big((A,\cdot),(B,\circ),l_A,r_A,l_B,r_B \big)$ is a matched pair of $A_3$-associative algebras if and only if $(l_A,r_A,B)$ is a representation of $(A,\cdot)$, $(l_B,r_B,A)$ is a representation of $(B,\circ)$ and for all $x,y,z \in A,~a,b,c \in B$, 
\begin{eqnarray} 
  && (r_B-l_B)(a)(x \cdot y) = x \cdot \big(r_B(a)y \big) - \big(r_B(a)y \big) \cdot x + y \cdot \big(l_B(a)x \big)  \notag \\
  && \quad - \big(l_B(a)x \big) \cdot y + (r_B-l_B)\big(l_A(y)a \big)x + (r_B-l_B)\big(r_A(x)a \big)y, \label{eq:mp1} \\
  && (r_A-l_A)(x)(a \circ b) = a \circ (r_A(x)b) - (r_A(x)b) \circ a + b \circ (l_A(x)a) \notag \\
  && \quad - (l_A(x)a) \circ b + (r_A-l_A)(l_B(b)x)a + (r_A-l_A)(r_B(a)x)b. \label{eq:mp2}
\end{eqnarray}
\end{proposition}
\begin{proof}
For all $x,y,z \in A,~a,b,c \in B$, we have 
\begin{eqnarray*}
  && \mathop{\circlearrowleft}\limits_{x,y,z}\mathop{\circlearrowleft}\limits_{a,b,c} \big((x+a) \ast (y+b) \big) \ast (z+c) \\
  & \overset{\eqref{eq:mp-ds}}{=}& \mathop{\circlearrowleft}\limits_{x,y,z}\mathop{\circlearrowleft}\limits_{a,b,c} \big(x \cdot y +l_B(a)y +r_B(b)x + a \circ b +l_A(x)b + r_A(y)a \big) \ast (z+c)  \\
  & \overset{\eqref{eq:mp-ds}}{=}& \mathop{\circlearrowleft}\limits_{x,y,z}\mathop{\circlearrowleft}\limits_{a,b,c} \big(x \cdot y +l_B(a)y +r_B(b)x \big) \cdot z + l_B \big(a \circ b +l_A(x)b +r_A(y)a \big)z \\
  &&+ r_B(c)\big(x \cdot y +l_B(a)y +r_B(b)x \big) + \big(a \circ b +l_A(x)b +r_A(y)a \big) \circ c \\
  &&+ l_A \big(x \cdot y +l_B(a)y +r_B(b)x \big)c + r_A(z)(a \circ b +l_A(x)b +r_A(y)a), \\
  && \mathop{\circlearrowleft}\limits_{x,y,z}\mathop{\circlearrowleft}\limits_{a,b,c}
  (x+a) \ast \big((y+b) \ast (z+c) \big) \\
  & \overset{\eqref{eq:mp-ds}}{=}& \mathop{\circlearrowleft}\limits_{x,y,z}\mathop{\circlearrowleft}\limits_{a,b,c} (x+a) \ast \big(y \cdot z +l_B(b)z +r_B(c)y + b \circ c +l_A(y)c + r_A(z)b \big) \\
  & \overset{\eqref{eq:mp-ds}}{=}& \mathop{\circlearrowleft}\limits_{x,y,z}\mathop{\circlearrowleft}\limits_{a,b,c} x \cdot \big(y \cdot z +l_B(b)z +r_B(c)y \big) +l_B(a) \big(y \cdot z +l_B(b)z +r_B(c)y \big) \\
  &&+ r_B \big(b \circ c +l_A(y)c +r_A(z)b \big)x + a \circ \big(b \circ c +l_A(y)c +r_A(z)b \big) \\
  &&+l_A(x) \big(b \circ c +l_A(y)c +r_A(z)b \big) + r_A \big(y \cdot z +l_B(b)z +r_B(c)y \big)a,
\end{eqnarray*}
where $\mathop{\circlearrowleft}\limits_{x,y,z}$ denotes the cyclic sum of $x,y,z$. Hence, 
$$\mathop{\circlearrowleft}\limits_{x,y,z}\mathop{\circlearrowleft}\limits_{a,b,c} \big((x+a) \ast (y+b)\big) \ast(z+c) = \mathop{\circlearrowleft}\limits_{x,y,z}\mathop{\circlearrowleft}\limits_{a,b,c} (x+a) \ast \big((y+b) \ast (z+c)\big)$$
holds if and only if $(l_A,r_A,B)$ is a representation of $(A,\cdot)$, $(l_B,r_B,A)$ is a representation of $(B,\circ)$, \eqref{eq:mp1} and \eqref{eq:mp2} hold.
Therefore, the conclusion holds.
\end{proof}

\begin{proposition} \label{semi-dp}
Let $(A,\cdot)$ be an $A_3$-associative algebra, $V$ be a vector space and $l,r: A\rightarrow End(V)$ be linear maps. Then $(l,r,V)$ is a representation of $(A,\cdot)$ if and only if there is an $A_3$-associative algebra on the direct sum $A\oplus V$ of vector spaces given by
\begin{equation} \label{eq:semi-dp}
  (x+u)\ast (y+v) = x \cdot y + l(x)v + r(y)u,\quad \forall~ x,y\in A,~u,v \in V.
\end{equation}
The $A_3$-associative algebra $(A\oplus V,\ast)$ is denoted by $A \ltimes_{l,r} V$, which is called the \textbf{semi-direct product $A_3$-associative algebra}.
\end{proposition}
\begin{proof}
It follows Proposition \ref{matched pair}.
\end{proof}

\begin{proposition} \label{adm-semi-dp}
Let $(A,\cdot)$ be an admissible $A_3$-associative algebra with an adjoint representation $(L,R,A)$. 
Then the semi-direct product $A_3$-associative algebra $A \ltimes_{L,R} A$ is admissible.
\end{proposition}
\begin{proof}
For all $x,y,z,d,e,f\in A$, we have 
\begin{eqnarray*}
  && \big((x+d) \ast (z+f)\big) \ast (y+e) - (x+d) \ast \big((z+f) \ast (y+e)\big) \\
  && - (y+e) \ast \big((z+f) \ast (x+d)\big) + \big((y+e) \ast (z+f)\big) \ast (x+d) \\
  &\overset{\eqref{eq:semi-dp}}{=}& \big(x \cdot z + L(x)f + R(z)d\big) \ast (y+e) - (x+d) \ast \big(z \cdot y + L(z)e + R(y)f\big)  \\
  && - (y+e) \ast \big(z \cdot x + L(z)d + R(x)f\big) + \big(y \cdot z + L(y)f + R(z)e\big) \ast (x+d) \\
  &\overset{\eqref{eq:semi-dp}}{=}& (x \cdot z) \cdot y + L(x\cdot z)e +R(y)\big(L(x)f + R(z)d\big) - x\cdot (z\cdot y) -L(x)\big(L(z)e + R(y)f\big)  \\
  && -R(z\cdot y)d - y\cdot (z\cdot x) -L(y)\big(L(z)d+R(x)f\big) -R(z\cdot x)e +(y\cdot z)\cdot x + L(y\cdot z)d   \\  
  && + R(x)\big(L(y)f + R(z)e\big) \\
  &\overset{\eqref{eq:ad-repre}\eqref{eq:admissible}}{=}& L(x\cdot z)e +R(y)R(z)d -L(x)L(z)e -R(z\cdot y)d -L(y)L(z)d -R(z\cdot x)e + L(y\cdot z)d \\
  && + R(x)R(z)e \\
  &=& (x\cdot z)\cdot e + (d\cdot z)\cdot y -x\cdot (z\cdot e) - d\cdot (z\cdot y) -y\cdot (z\cdot d) -e\cdot (z\cdot x) + (y\cdot z)\cdot d \\
  && + (e\cdot z)\cdot x  \\
  &\overset{\eqref{eq:admissible}}{=}& 0.
\end{eqnarray*}
Thus, the semi-direct product $A_3$-associative algebra $A \ltimes_{L,R} A$ is admissible.
\end{proof}

\subsection{Quadratic $A_3$-associative algebras}

\begin{definition}
A \textbf{quadratic $A_3$-associative algebra} is a triple $(A,\cdot,\mathcal{B})$, where $(A,\cdot)$ is an $A_3$-associative algebra, $\mathcal{B}: A \times A \rightarrow \mathbb{K}$ is a nondegenerate symmetric bilinear form on $A$ which is invariant, that is, satisfying
\begin{equation} \label{eq:invar}
  \mathcal{B}(x\cdot y,z) = \mathcal{B}(x,y\cdot z),\quad \forall~ x,y,z\in A.
\end{equation}
\end{definition}

Let $(A,\cdot)$ be an admissible $A_3$-associative algebra, $(L,R,A)$ be an adjoint representation of $(A,\cdot)$, $L^*,R^*:A \rightarrow End(A^*)$ be two linear maps. The dual actions of $(A,\cdot)$ are defined as
\begin{equation} \label{eq:ad-dual}
   \langle L^*(x)a^*,y \rangle = \langle a^*,L(x)y \rangle,~ \langle R^*(x)a^*,y \rangle = \langle a^*,R(x)y \rangle, \quad \forall~ x,y\in A,~a^* \in A^*.
\end{equation}
Then by Proposition \ref{co-rep}, $(R^*,L^*,A^*)$ is a representation. 

\begin{lemma}
Let $(A,\cdot,\mathcal{B})$ be a quadratic $A_3$-associative algebra, where $(A,\cdot)$ is an admissible $A_3$-associative algebra. Then the adjoint representation $(L,R,A)$ and the coadjoint representation $(R^*,L^*,A^*)$ are equivalent. 
\end{lemma}
\begin{proof}
Let $\mathcal{B}^\natural: A \rightarrow A^*$ be a linear isomorphism given by 
\begin{equation} \label{eq:lin-iso}
  \langle \mathcal{B}^\natural(x),y \rangle=\mathcal{B}(x,y), \quad \forall~ x,y\in A.
\end{equation}
Since $\mathcal{B}$ is symmetric and by the invariant condition \eqref{eq:invar}, we have
\begin{equation} \label{eq:circle-invar}
   \mathcal{B}(x\cdot y,z) = \mathcal{B}(z,x\cdot y) = \mathcal{B}(z\cdot x, y)=  \mathcal{B}(y,z\cdot x),\quad \forall~z\in A.
\end{equation}
Then
\begin{equation*}
  \langle \mathcal{B}^\natural(L(x)y),z \rangle \overset{\eqref{eq:lin-iso}}{=} \mathcal{B}(x \cdot y,z) \overset{\eqref{eq:circle-invar}}{=} \mathcal{B}(y,z\cdot x) \overset{\eqref{eq:lin-iso}}{=} \langle \mathcal{B}^\natural(y),z\cdot x \rangle \overset{\eqref{eq:ad-dual}}{=} \langle R^*(x) \mathcal{B}^\natural (y),z\rangle,
\end{equation*}
that is, $\mathcal{B}^\natural(L(x))y = R^*(x) \mathcal{B}^\natural (y)$.
Similarly, we have $\mathcal{B}^\natural(R(x))y = L^*(x) \mathcal{B}^\natural (y)$. If we take $\phi=\mathcal{B}^\natural$, then by Definition \ref{re-equivalent}, the conclusion holds.
\end{proof}

Let $(L,R,A)$ be an adjoint representation of an admissible $A_3$-associative algebra $(A,\cdot)$. By Proposition \ref{semi-dp}, the multiplication $\ast_\mathfrak{d}: (A \oplus A^*) \times (A \oplus A^*) \rightarrow A \oplus A^*$ on $\mathfrak{d}=A \oplus A^*$ is defined as
\begin{equation} \label{eq:ad-semi-dp}
  (x+a^*)\ast_\mathfrak{d} (y+b^*) = x \cdot y + L(x)b^* + R(y)a^*,\quad \forall~ x,y\in A,~a^*,b^* \in A^*.
\end{equation}

\begin{example}
Let $(A,\cdot)$ be an admissible $A_3$-associative algebra and $(L,R,A)$ be the adjoint representation. By Proposition \ref{co-rep},
there exists a natural nondegenerate symmetric bilinear form $\mathcal{B}_\mathfrak{d}$ on $\mathfrak{d}$ defined as
\begin{equation} \label{eq:sym-bilf}
  \mathcal{B}_\mathfrak{d}(x+a^*,y+b^*) = \langle b^*,x \rangle + \langle a^*,y \rangle, \quad \forall~ x,y\in A, ~ a^*,b^*\in A^*.
\end{equation}
In fact, for all $x,y,z \in A$, $a^*,b^*,c^* \in A^*$, we have
\begin{align*}
  \mathcal{B}_\mathfrak{d}\big((x+a^*) \cdot_d (y+b^*),z+c^* \big)
  \overset{\eqref{eq:ad-semi-dp}}{=}& \mathcal{B}_\mathfrak{d} \big(x \cdot y + R^*(x)b^* + L^*(y)a^*,z+c^* \big) \\
  \overset{\eqref{eq:sym-bilf}}{=}& \langle c^*,x\cdot y \rangle + \langle R^*(x)b^* + L^*(y)a^*, z \rangle \\
  \overset{\eqref{eq:ad-dual}}{=}& \langle c^*,x\cdot y \rangle + \langle b^*,z \cdot x \rangle + \langle a^*,y \cdot z \rangle, \\
  \overset{\eqref{eq:ad-dual}}{=}& \langle R^*(y)c^* + L^*(z)b^*,x \rangle + \langle a^*, y \cdot z \rangle \\
  \overset{\eqref{eq:sym-bilf}}{=}& \mathcal{B}_\mathfrak{d} \big(x+a^*, y \cdot z + R^*(y)c^* + L^*(z)b^* \big) \\
  \overset{\eqref{eq:ad-semi-dp}}{=}& \mathcal{B}_\mathfrak{d} \big(x+a^*, (y+b^*) \cdot_d (z+c^*) \big).
\end{align*}
Then $\mathcal{B}_\mathfrak{d}$ is invariant, that is, $(A \ltimes_{R^*,L^*} A^*, \cdot_\mathfrak{d}, \mathcal{B}_\mathfrak{d})$ is a quadratic $A_3$-associative algebra.
\end{example}

\begin{definition} \label{te-form}
Let $(A,\cdot)$ be an $A_3$-associative algebra and $r \in A \otimes A$. Let $h:A \rightarrow End(A \otimes A)$ be a linear map defined by
\begin{equation} \label{eq:hx}
  h(x) = id \otimes L(x) - R(x) \otimes id, \quad \forall~ x\in A,
\end{equation}
where $id$ denotes the identity map. The element $r$ is called \textbf{invariant} on $(A,\cdot)$ if $h(x)r=0$ for all $x \in A$.  
\end{definition}

Let $r=\sum\limits_{i} (u_{i}\otimes v_{i}) \in A\otimes A$, then $h(x)r= \sum\limits_{i} (u_{i}\otimes x\cdot v_{i} -u_{i} \cdot x\otimes v_{i}) =0$, which is equivalent to $\mathcal{B}(x,y\cdot z) -\mathcal{B}(x\cdot y,z) =0$.
Thus, Definition \ref{te-form} gives the tensor forms of invariant bilinear forms on $A_3$-associative algebras.

\begin{proposition}
Let $(A,\cdot)$ be an admissible $A_3$-associative algebra and $\mathcal{B}: A \times A \rightarrow \mathbb{K}$ be a nondegenerate bilinear form. Let $\mathcal{B}^\natural: A \rightarrow A^*$ be a linear isomorphism given by \eqref{eq:lin-iso}. If $\widetilde{\mathcal{B}} \in A \otimes A$ satisfies
\begin{equation} \label{eq:tensor-liniso}
  \langle \widetilde{\mathcal{B}},a^* \otimes b^* \rangle = \langle {\mathcal{B}^\natural}^{-1} (a^*),b^* \rangle, \quad \forall~ a^*,b^* \in A^*,
\end{equation}
then $(A,\cdot,\mathcal{B})$ is a quadratic $A_3$-associative algebra if and only if $\widetilde{\mathcal{B}}$ is symmetric and invariant.
\end{proposition}
\begin{proof}
By \eqref{eq:lin-iso}, we have $\mathcal{B}$ is symmetric if and only if $\widetilde{\mathcal{B}}$ is symmetric. 
For all $x,y,z\in A$, let $a^*={\mathcal{B}^\natural}(x),b^*={\mathcal{B}^\natural}(y)$, we have
\begin{eqnarray*}
  && \mathcal{B}(x,y\cdot z)
  \overset{\eqref{eq:lin-iso}}{=} \langle \mathcal{B}^\natural(x), y\cdot z \rangle
  = \langle a^*,{\mathcal{B}^\natural}^{-1}(b^*)\cdot z \rangle
  \overset{\eqref{eq:ad-dual}}{=} \langle R^*(z)a^*, {\mathcal{B}^\natural}^{-1}(b^*) \rangle  \\
  &\overset{\eqref{eq:tensor-liniso}}{=}& \langle \widetilde{\mathcal{B}}, b^* \otimes R^*(z)a^* \rangle
  = \langle \big(R(z) \otimes id \big)\widetilde{\mathcal{B}}, a^* \otimes b^* \rangle,  \\
  && \mathcal{B}(y, z\cdot x)
  \overset{\eqref{eq:lin-iso}}{=} \langle \mathcal{B}^\natural(y), z\cdot x  \rangle
  = \langle b^*,z \cdot {\mathcal{B}^\natural}^{-1} (a^*) \rangle
  \overset{\eqref{eq:ad-dual}}{=} \langle L^*(z)b^*, {\mathcal{B}^\natural}^{-1}(a^*) \rangle  \\
  &\overset{\eqref{eq:tensor-liniso}}{=}& \langle \widetilde{\mathcal{B}}, a^* \otimes L^*(z)b^* \rangle
  = \langle \big(id \otimes L(z) \big)\widetilde{\mathcal{B}}, a^* \otimes b^* \rangle.
\end{eqnarray*}
Then, $(A,\cdot,\mathcal{B})$ is a quadratic $A_3$-associative algebra  if and only if $h(y)\widetilde{\mathcal{B}}=0$ for all $y \in A$. Thus, the conclusion follows.
\end{proof}

\section{Manin triples of $A_3$-associative algebras and $A_3$-associative bialgebras}

In this section, we introduce the concepts of Manin triples of $A_3$-associative algebras and $A_3$-associative bialgebras. We show that they are equivalent by matched pairs of $A_3$-associative algebras.

\subsection{Manin triples of $A_3$-associative algebras}

\begin{definition} \label{Manin triples}
Let $(A,\cdot)$ and $(A',\circ)$ be $A_3$-associative algebras. The triple $\big((A \oplus A'=\mathfrak{d},\cdot_\mathfrak{d},\mathcal{B}_\mathfrak{d}), A, A'\big)$ is called a \textbf{Manin triple of $A_3$-associative algebras} where
\begin{enumerate}[(1)]
  \item $(A \oplus A',\cdot_\mathfrak{d},\mathcal{B}_\mathfrak{d})$ is a quadratic $A_3$-associative algebra;
  \item $(A,\cdot)$ and $(A',\circ)$ are $A_3$-associative subalgebras of $(A \oplus A',\cdot_\mathfrak{d})$, such that $\mathfrak{d}=A \oplus A'$ as vector spaces;
  \item $(A,\cdot)$ and $(A',\circ)$ are isotropic, that is, for all $x,y\in A$, $a',b'\in A'$, $\mathcal{B}_\mathfrak{d}(x,y) =\mathcal{B}_\mathfrak{d}(a',b') =0$. 
\end{enumerate}
\end{definition}

To distinguish the left (or right) multiplications of different algebras, we denote the left (or right) multiplications on $(A,\cdot)$ as $L_\cdot$ (or $R_\cdot$), the left (or right) multiplications on $(A^*,\circ)$ as $L_\circ$ (or $R_\circ$).
\begin{proposition} \label{Mant-mp}
Let $(A,\cdot)$ and $(A^*,\circ)$ be admissible $A_3$-associative algebras, and let $(R_\cdot^*,L_\cdot^*,A^*)$ and $(R_\circ^*,L_\circ^*,A)$ be the coadjoint representations of $(A,\cdot)$ and $(A^*,\circ)$, respectively. 
Then there is a Manin triple of $A_3$-associative algebras $\big((A \oplus A^*=\mathfrak{d}, \cdot_\mathfrak{d}, \mathcal{B}_\mathfrak{d}),A,A^*\big)$ if and only if there is a matched pair of $A_3$-associative algebras $\big((A,\cdot),(A^*,\circ),R^*_\cdot,L^*_\cdot, R^*_\circ,L^*_\circ\big)$.
\end{proposition}
\begin{proof}
Let $\big((\mathfrak{d},\cdot_\mathfrak{d},\mathcal{B}_\mathfrak{d}),A,A^*\big)$ be a Manin triple of $A_3$-associative algebras, where the nondegenerate symmetric bilinear form $\mathcal{B}_\mathfrak{d}$ on $\mathfrak{d}$ is given by \eqref{eq:sym-bilf}. For all $x,y \in A$, $a^*,b^* \in A^*$, we have
\begin{align*}
  \mathcal{B}_\mathfrak{d}(x \cdot_\mathfrak{d} b^*,y)
  &\overset{\eqref{eq:circle-invar}\eqref{eq:ad-semi-dp}}{=} \mathcal{B}_\mathfrak{d}(y \cdot x , b^*)
  \overset{\eqref{eq:sym-bilf}}{=} \langle y \cdot x, b^* \rangle
  \overset{\eqref{eq:ad-dual}}{=} \langle y, R^*_\cdot(x)b^* \rangle
  \overset{\eqref{eq:sym-bilf}}{=} \mathcal{B}_\mathfrak{d} \big(R^*_\cdot(x)b^*,y \big),\\
  \mathcal{B}_\mathfrak{d}(x \cdot_\mathfrak{d} b^*, a^*)
  &\overset{\eqref{eq:circle-invar}\eqref{eq:ad-semi-dp}}{=} \mathcal{B}_\mathfrak{d}(x, b^* \circ a^*)
  \overset{\eqref{eq:sym-bilf}}{=} \langle x, b^* \circ a^* \rangle
  \overset{\eqref{eq:ad-dual}}{=} \langle L^*_\circ(b^*)x, a^* \rangle
  \overset{\eqref{eq:sym-bilf}}{=} \mathcal{B}_\mathfrak{d} \big(L^*_\circ(b^*)x, a^* \big).
\end{align*}
Thus,
\begin{equation*}
  \mathcal{B}_\mathfrak{d}(x \cdot_\mathfrak{d} b^*, y+a^*)
  = \mathcal{B}_\mathfrak{d} \big(R^*_\cdot(x)b^* + L^*_\circ(b^*)x, y+a^* \big).
\end{equation*}
Since $\mathcal{B}_\mathfrak{d}$ is nondegenerate, we have
\begin{equation*}
  x \cdot_\mathfrak{d} b^* = R^*_\cdot(x)b^* + L^*_\circ(b^*)x.
\end{equation*}
Similarly, we have
\begin{equation*}
  a^* \cdot_\mathfrak{d} y = L^*_\cdot(y)a^* + R^*_\circ(a^*)y.
\end{equation*}
Then
\begin{eqnarray} 
  (x + a^*) \cdot_\mathfrak{d} (y + b^*) &=& x\cdot y + x \cdot_\mathfrak{d} b^* + a^* \cdot_\mathfrak{d} y + a^* \circ b^* \label{eq:mpMt} \\ 
  &=& x\cdot y + R^*_\circ(a^*)y + L^*_\circ(b^*)x + a^* \circ b^* + R^*_\cdot(x)b^* + L^*_\cdot(y)a^*. \notag
\end{eqnarray}
Thus, $\big((A,\cdot),(A^*,\circ),R^*_\cdot,L^*_\cdot,R^*_\circ,L^*_\circ \big)$ is a matched pair of $A_3$-associative algebra.

Conversely, let $\big((A,\cdot),(A^*,\circ),R^*_\cdot,L^*_\cdot, R^*_\circ,L^*_\circ\big)$ be a matched pair of $A_3$-associative algebras, then by \eqref{eq:mpMt}, we have

(1) $\mathcal{B}_\mathfrak{d}$ is invariant on the $A_3$-associative algebra $(A \oplus A^*,\cdot_\mathfrak{d})$. Then $(A \oplus A^*,\cdot_\mathfrak{d},\mathcal{B}_\mathfrak{d})$ is a quadratic $A_3$-associative algebra;

(2) $(A,\cdot)$ and $(A^*,\circ)$ are $A_3$-associative subalgebras of $(A \oplus A^*,\cdot_\mathfrak{d})$, such that $\mathfrak{d}=A \oplus A^*$ as vector spaces;

(3) for all $x,y\in A$, $a^*,b^*\in A^*$, we have $\mathcal{B}_\mathfrak{d}(x,y) =\mathcal{B}_\mathfrak{d}(a^*,b^*) =0$. 

Therefore, $\big((A \oplus A^*=\mathfrak{d}, \cdot_\mathfrak{d}, \mathcal{B}_\mathfrak{d}),A,A^*\big)$ is a Manin triple of $A_3$-associative algebras if and only if $\big((A,\cdot),(A^*,\circ),R^*_\cdot,L^*_\cdot, R^*_\circ,L^*_\circ\big)$ is a matched pair of $A_3$-associative algebras.
\end{proof}

\subsection{$A_3$-associative bialgebras}

\begin{definition}
An \textbf{$A_3$-associative coalgebra} is a pair $(A,\Delta)$, where $A$ is a vector space and $\Delta:A \rightarrow A \otimes A$ is a linear map satisfying
\begin{equation} \label{eq:asso-coalg}
  (id^{\otimes3} + \xi + \xi^2)(\Delta \otimes id - id \otimes \Delta)\Delta =0,
\end{equation}
where $\xi(x\otimes y\otimes z)=y\otimes z\otimes x$ for all $x,y,z \in A$.
\end{definition}

\begin{proposition}
Let $A$ be a vector space, $\Delta:A \rightarrow A \otimes A$ be a linear map. Let $\circ: A^* \otimes A^* \rightarrow A^*$ be a linear map satisfying
\begin{equation} \label{eq:co-mul}
  \langle a^* \circ b^*,x \rangle = \langle a^* \otimes b^*,\Delta(x) \rangle, \quad \forall~ a^*,b^* \in A^*, x\in A.
\end{equation}
Then $(A,\Delta)$ is an $A_3$-associative coalgebra if and only if $(A^*,\circ)$ is an $A_3$-associative algebra. 
\end{proposition}
\begin{proof}
Let 
\begin{equation} \label{eq:co-delta}
\Delta^*(a^* \otimes b^*)=  a^* \circ b^*,\quad \forall~ a^*,b^* \in A^*.
\end{equation}
Using the Sweedler notation, we denote $\Delta(x):=x_1\otimes x_2$, then we have
\begin{eqnarray*}
  && \langle (\Delta^*\otimes id)(a^*\otimes b^*\otimes c^*), \Delta(x) \rangle \\
  &=& \langle \Delta^*(a^*\otimes b^*) \otimes c^*, \Delta(x) \rangle  \\
  &=& \langle \Delta^*(a^*\otimes b^*), x_1 \rangle \langle c^*, x_2 \rangle \\
  &\overset{\eqref{eq:co-mul}}{=}& \langle a^*\otimes b^*, \Delta(x_1) \rangle \langle c^*, x_2 \rangle  \\
  &=& \langle a^*\otimes b^*\otimes c^*, \Delta(x_1) \otimes x_2 \rangle \\
  &=& \langle a^*\otimes b^*\otimes c^*, (\Delta \otimes id) \Delta(x) \rangle.
\end{eqnarray*}
For all $x \in A, a^*,b^*,c^* \in A^*$, we have
\begin{eqnarray*}
  \langle (a^* \circ b^*) \circ c^* - a^* \circ (b^* \circ c^*), x \rangle
  &=& \langle \big(\Delta^*(\Delta^*\otimes id)-\Delta^*(id \otimes \Delta^*)\big) (a^*\otimes b^*\otimes c^*),x \rangle  \\
  &\overset{\eqref{eq:co-mul}}{=}& \langle a^*\otimes b^*\otimes c^*,(\Delta\otimes id - id \otimes \Delta) \Delta(x)\rangle,  \\
  \langle(b^* \circ c^*) \circ a^* - b^* \circ (c^* \circ a^*), x \rangle
  &=& \langle \big(\Delta^*(\Delta^*\otimes id)-\Delta^*(id \otimes \Delta^*) \big) (b^*\otimes c^*\otimes a^*), x \rangle  \\
  &\overset{\eqref{eq:co-mul}}{=}& \langle a^*\otimes b^*\otimes c^*, \xi^2(\Delta\otimes id -id \otimes \Delta) \Delta(x)\rangle,  \\
  \langle(c^* \circ a^*) \circ b^* - c^* \circ (a^* \circ b^*), x \rangle
  &=& \langle \big(\Delta^*(\Delta^*\otimes id)-\Delta^*(id \otimes \Delta^*) \big) (c^*\otimes a^*\otimes b^*), x \rangle  \\
  &\overset{\eqref{eq:co-mul}}{=}& \langle a^*\otimes b^*\otimes c^*, \xi(\Delta\otimes id -id \otimes \Delta) \Delta(x) \rangle.
\end{eqnarray*}
Therefore, \eqref{eq:asso-alg} holds if and only if \eqref{eq:asso-coalg} holds. Thus, the conclusion follows.
\end{proof}

\begin{remark}
Let $A$ be an associative algebra, and let $\Delta:A \rightarrow A \otimes A$ satisfy the coassociativity condition
\begin{equation} \label{co-coasso}
  (\Delta \otimes id) \Delta = (id \otimes \Delta)\Delta.
\end{equation}
Then $(A,\Delta)$ is an coassociative coalgebra. Moreover, the coassociative coalgebra is an $A_3$-associative coalgebra.
\end{remark}

\begin{definition} \label{bialgebra}
An \textbf{$A_3$-associative bialgebra} is a triple $(A,\cdot,\Delta)$, where $(A,\cdot)$ is an $A_3$-associative algebra, $(A,\Delta)$ is an $A_3$-associative coalgebra and for all $x,y \in A$,
\begin{eqnarray}
  && \quad \quad \quad (\tau-id^{\otimes2}) \big(\Delta(x \cdot y) -(R(y)\otimes id)\Delta(x) -(id\otimes L(x))\Delta(y) \big) \notag \\
  && + \big(id \otimes L(y)-R(y)\otimes id \big) \tau\Delta(x) + \big(L(x) \otimes id -id\otimes R(x) \big)\Delta(y) =0,  \label{eq:asso-bialg} \\
  && \quad \quad \Delta(x \cdot y -y \cdot x) + \big((\tau+id^{\otimes2})(id \otimes L(y)-R(y)\otimes id) \big) \Delta(x) \notag \\
  && \quad \quad \quad \quad \quad -\big(id \otimes L(x) -R(x)\otimes id \big) \big(\Delta(y)- \tau\Delta(y) \big) =0.  \label{eq:asso-bialg2}
\end{eqnarray}
\end{definition}

\begin{definition}
An \textbf{admissible $A_3$-associative coalgebra} is an $A_3$-associative coalgebra $(A,\Delta)$ satisfying
\begin{equation} \label{eq:ad-asso-coalg}
  \xi(\tau\otimes id)(\Delta \otimes id)\Delta - (id\otimes \tau)(id \otimes \Delta)\Delta + \xi^2 (\Delta \otimes id - id \otimes \Delta)\Delta =0,
\end{equation}
where $\tau(x \otimes y)=y\otimes x$ for all $x,y,z \in A$.
\end{definition}

\begin{proposition}
Let $A$ be a vector space, $\Delta:A \rightarrow A \otimes A$ be a linear map. Let $\circ: A^* \otimes A^* \rightarrow A^*$ be a linear map satisfying
\eqref{eq:co-mul}. Then $(A,\Delta)$ is an admissible $A_3$-associative coalgebra if and only if $(A^*,\circ)$ is an admissible $A_3$-associative algebra.
\end{proposition}
\begin{proof}
For all $x \in A, a^*,b^*,c^* \in A^*$, we have
\begin{eqnarray*}
  \langle(b^* \circ c^*) \circ a^* - b^* \circ (c^* \circ a^*), x \rangle
  &=& \langle \big(\Delta^*(\Delta^*\otimes id)-\Delta^*(id \otimes \Delta^*) \big) (b^*\otimes c^*\otimes a^*), x \rangle  \\
  &\overset{\eqref{eq:co-mul}}{=}& \langle a^*\otimes b^*\otimes c^*, \xi^2(\Delta\otimes id -id \otimes \Delta) \Delta(x)\rangle,  \\
  \langle (a^* \circ c^*) \circ b^*, x \rangle
  &=& \langle \big(\Delta^*(\Delta^*\otimes id)\big) (a^*\otimes c^*\otimes b^*), x \rangle  \\
  &=& \langle \big(\Delta^*(\Delta^*\otimes id)\big)(\tau \otimes id) (c^*\otimes a^*\otimes b^*), x \rangle  \\
  &\overset{\eqref{eq:co-mul}}{=}& \langle a^*\otimes b^*\otimes c^*, \xi(\tau \otimes id)(\Delta\otimes id) \Delta(x)\rangle, \\
  \langle a^* \circ (c^* \circ b^*), x \rangle 
  &=& \langle \big(\Delta^*(id \otimes \Delta^*)\big) (a^*\otimes c^*\otimes b^*), x \rangle  \\
  &=& \langle \big(\Delta^*(id \otimes \Delta^*)\big)(id \otimes \tau) (a^*\otimes b^*\otimes c^*), x \rangle  \\
  &\overset{\eqref{eq:co-mul}}{=}& \langle a^*\otimes b^*\otimes c^*, (id \otimes \tau)(id \otimes \Delta) \Delta(x)\rangle,
\end{eqnarray*}
then \eqref{eq:admissible} holds if and only if \eqref{eq:ad-asso-coalg} holds. Thus, the conclusion follows.
\end{proof}

\begin{lemma} \label{mp-bialg}
Let $(A,\cdot)$ and $(A^*,\circ)$ be admissible $A_3$-associative algebras, and let $(R_\cdot^*,L_\cdot^*,A)$ and $(R_\circ^*,L_\circ^*,A^*)$ be the coadjoint representations of $(A,\cdot)$ and $(A^*,\circ)$, respectively.
Let $\Delta: A \rightarrow A \otimes A$ be a linear map satisfying \eqref{eq:co-mul}.
Then $\big((A,\cdot),(A^*,\circ), R^*_\cdot, L^*_\cdot, R^*_\circ, L^*_\circ \big)$ is a matched pair of $A_3$-associative algebra if and only if $(A,\cdot,\Delta)$ is an $A_3$-associative bialgebra.
\end{lemma}
\begin{proof}
For all $x,y \in A, a^*,b^* \in A^*$, by \eqref{eq:ad-dual} and \eqref{eq:co-mul}, we have
\begin{eqnarray*}
  \langle (R^*_\circ-L^*_\circ)(a^*)(x \cdot y),b^* \rangle
  &=& \langle x \cdot y, b^* \circ a^* -a^* \circ b^* \rangle  \\
  &=& \langle(\tau-id^{\otimes2})\Delta(x \cdot y), a^* \otimes b^* \rangle, \\
  \langle x \cdot (L_\circ^*(a^*)y),b^* \rangle
  &=& \langle L^*_\circ(a^*)y, L^*_\cdot(x)b^* \rangle \\
  &=& \langle y,a^* \circ (L^*_\cdot(x)b^*) \rangle \\
  &=& \langle (id \otimes L_\cdot(x)) \Delta(y), a^* \otimes b^* \rangle,\\
  -\langle (L^*_\circ(a^*)y) \cdot x,b^* \rangle
  &=& -\langle L^*_\circ(a^*)y, R^*_\cdot(x)b^* \rangle \\
  &=& -\langle y, a^* \circ (R^*_\cdot(x)b^*) \rangle \\
  &=& -\langle (id \otimes R_\cdot(x)) \Delta(y), a^* \otimes b^* \rangle,\\
  \langle y \cdot (R^*_\circ(a^*)x),b^* \rangle
  &=& \langle R^*_\circ(a^*)x, L^*_\cdot(y)b^* \rangle \\
  &=& \langle x, (L^*_\cdot(y)b^*) \circ a^* \rangle \\
  &=& \langle \tau(L_\cdot(y)\otimes id) \Delta(x), a^* \otimes b^* \rangle, \\
  -\langle (R^*_\circ(a^*)x) \cdot y,b^* \rangle
  &=& -\langle R^*_\circ(a^*)x, R^*_\cdot(y)b^* \rangle \\
  &=& -\langle x, (R^*_\cdot(y)b^*) \circ a^* \rangle \\
  &=& -\langle \tau(R_\cdot(y)\otimes id)\Delta(x), a^* \otimes b^* \rangle,\\
  \langle (L^*_\circ-R^*_\circ)(R^*_\cdot(y)a^*)x, b^* \rangle
  &=& \langle x, (R^*_\cdot(y)a^*) \circ b^* - b^* \circ (R^*_\cdot(y)a^*) \rangle\\
  &=& \langle (R_\cdot(y)\otimes id - \tau(id \otimes R_\cdot(y)))\Delta(x), a^* \otimes b^* \rangle, \\
  \langle (L^*_\circ-R^*_\circ)(L^*_\cdot(x)a^*)y, b^* \rangle
  &=& \langle y, (L^*_\cdot(x)a^*) \circ b^* - b^* \circ (L^*_\cdot(x)a^*) \rangle\\
  &=& \langle (L_\cdot(x)\otimes id - \tau(id \otimes L_\cdot(x)))\Delta(y), a^* \otimes b^* \rangle.
\end{eqnarray*}
Thus, by adding the above equations, \eqref{eq:asso-bialg} holds if and only if \eqref{eq:mp1} holds for $l_A=R^*_\cdot, r_A=L^*_\cdot, l_B=R^*_\circ$, and $r_B=L^*_\circ$.  Moreover, we have
\begin{eqnarray*}
  -\langle (L^*_\cdot-R^*_\cdot)(x)(a^* \circ b^*), y \rangle
  &=& -\langle a^* \circ b^*, x \cdot y -y \cdot x \rangle  \\
  &=& -\langle a^* \otimes b^*, \Delta(x \cdot y -y \cdot x) \rangle,  \\
  \langle a^* \circ (L_\cdot^*(x)b^*), y \rangle
  &=& \langle a^* \otimes b^*, (id \otimes L_\cdot(x)) \Delta(y) \rangle, \\
  -\langle (L_\cdot^*(x)b^*) \circ a^*, y \rangle
  &=& -\langle a^* \otimes b^*, \tau(L_\cdot(x)\otimes id)\Delta(y) \rangle, \\
  \langle b^* \circ (R_\cdot^*(x)a^*),y \rangle
  &=& \langle a^* \otimes b^*, \tau(id \otimes R_\cdot(x))\Delta(y) \rangle, \\
  -\langle (R_\cdot^*(x)a^*) \circ b^*,y \rangle
  &=& -\langle a^* \otimes b^*, (R_\cdot(x)\otimes id)\Delta(y) \rangle, \\
  \langle (L_\cdot^*-R_\cdot^*)(R_\circ^*(b^*)x)a^*, y \rangle
  &=& \langle a^*, (R_\circ^*(b^*)x) \cdot y - y \cdot (R_\circ^*(b^*)x) \rangle\\
  &=& \langle R_\cdot^*(y)a^* - L_\cdot^*(y)a^*, R_\circ^*(b^*)x \rangle \\
  &=& \langle (R_\cdot^*(y)a^*) \circ b^* - (L_\cdot^*(y)a^*) \circ b^*, x \rangle \\
  &=& \langle a^* \otimes b^*, (R_\cdot(y)\otimes id)\Delta(x) - (L_\cdot(y)\otimes id)\Delta(x) \rangle, \\
  \langle (L_\cdot^*-R_\cdot^*)(L_\circ^*(a^*)x)b^*, y \rangle
  &=& \langle b^*, (L_\circ^*(a^*)x) \cdot y - y \cdot (L_\circ^*(a^*)x) \rangle\\
  &=& \langle R_\cdot^*(y)b^* - L_\cdot^*(y)b^*, L_\circ^*(a^*)x \rangle \\
  &=& \langle a^* \circ (R_\cdot^*(y)b^*) - a^* \circ (L_\cdot^*(y)b^*), x \rangle \\
  &=& \langle a^* \otimes b^*, (id \otimes R_\cdot(y)) \Delta(x) - (id \otimes L_\cdot(y))\Delta(x) \rangle.
\end{eqnarray*}
Thus, by adding the above equations, \eqref{eq:asso-bialg2} holds if and only if \eqref{eq:mp2} holds for $l_A=R^*_\cdot, r_A=L^*_\cdot, l_B=R^*_\circ$, and $r_B=L^*_\circ$. 
Therefore, $((A,\cdot),(A^*,\circ), R^*_\cdot, L^*_\cdot, R^*_\circ, L^*_\circ)$ is a matched pair of $A_3$-associative algebra if and only if $(A,\cdot,\Delta)$ is an $A_3$-associative bialgebra.
\end{proof}

\begin{theorem} \label{equivalent}
Let $(A,\cdot)$ and $(A^*,\circ)$ be two admissible $A_3$-associative algebras, 
$\Delta: A \rightarrow A \otimes A$ be a linear map satisfying \eqref{eq:co-mul}. Then the following conditions are equivalent:
\begin{enumerate}[(1)]
  \item  There is a Manin triple of $A_3$-associative algebras $\big((A \oplus A^*,\cdot_\mathfrak{d},\mathcal{B}_\mathfrak{d}),A,A^* \big)$.
  \item $\big((A,\cdot),(A^*,\circ), R^*_\cdot,L^*_\cdot,R^*_\circ,L^*_\circ \big)$ is a matched pair of $A_3$-associative algebras.
  \item $(A,\cdot,\Delta)$ is an $A_3$-associative bialgebra.
\end{enumerate}
\end{theorem}
\begin{proof}
By Proposition \ref{Mant-mp}, we have $(1)\Leftrightarrow (2)$. By Lemma \ref{mp-bialg}, we have $(2)\Leftrightarrow (3)$. Hence, the conclusion holds.
\end{proof}

\begin{example}
Let $(A,\cdot)$ be the admissible $A_3$-associative algebra given in Example \ref{ad-asso-example}, $\Delta: A \rightarrow A \otimes A$ be a linear map on $A$ defined as
\begin{equation} \label{co-exam}
  \Delta(e_1)=\Delta(e_1+2e_2) = e_1 \otimes e_1,  \Delta(e_2)=0.
\end{equation}
Then $(A,\cdot,\Delta)$ is an $A_3$-associative bialgebra.
\end{example}
\begin{proof}
Firstly, by \eqref{eq:asso-coalg}, we have
\begin{eqnarray*}
  && (id^{\otimes3} +\xi +\xi^2)(\Delta \otimes id -id\otimes \Delta)\Delta(e_1) \\
  &=& (id^{\otimes3} +\xi +\xi^2) \big( (e_1 \otimes e_1) \otimes e_1 -e_1 \otimes (e_1 \otimes e_1) \big) \\
  &=& 0.
\end{eqnarray*}
Then $(A,\Delta)$ is an $A_3$-associative coalgebra. Moreover, 
\begin{eqnarray*}
  && \xi(\tau\otimes id)(\Delta \otimes id)\Delta (e_1) - (id\otimes \tau)(id \otimes \Delta)\Delta (e_1) + \xi^2 (\Delta \otimes id - id \otimes \Delta)\Delta(e_1) \\
  &=& \xi(\tau\otimes id)(\Delta (e_1) \otimes e_1) - (id\otimes \tau)\big(e_1 \otimes \Delta(e_1)\big) + \xi^2 \big(\Delta (e_1) \otimes e_1 - e_1 \otimes \Delta(e_1)\big) \\
  &=& (e_1 \otimes e_1 \otimes e_1) - (e_1 \otimes e_1 \otimes e_1) + (e_1 \otimes e_1 \otimes e_1) - (e_1 \otimes e_1 \otimes e_1).
\end{eqnarray*}
Then $(A,\Delta)$ is an admissible $A_3$-associative coalgebra.

Secondly, by \eqref{eq:asso-bialg} and \eqref{eq:asso-bialg2}, we have
\begin{eqnarray*}
  && (\tau-id^{\otimes2}) \big(\Delta(e_1 \cdot e_2) -(R(e_2)\otimes id)\Delta(e_1) -(id\otimes L(e_1))\Delta(e_2)\big) \\
  && +\big(id \otimes L(e_2)-R(e_2)\otimes id\big) \tau\Delta(e_1) +\big(L(e_1) \otimes id -id\otimes R(e_1)\big)\Delta(e_2) \\
  &=& (\tau-id^{\otimes2}) \big(e_1 \otimes e_1 -(R(e_2)\otimes id)(e_1 \otimes e_1)\big) + \big(id \otimes L(e_2)-R(e_2)\otimes id \big) (e_1 \otimes e_1) \\
  &=& (\tau-id^{\otimes2}) \big(e_1 \otimes e_1 -(e_1+2e_2) \otimes e_1 \big) + e_1 \otimes (e_1+2e_2) - (e_1+2e_2) \otimes e_1 \\
  &=& 0,
\end{eqnarray*}
and
\begin{eqnarray*}
  && \Delta(e_1 \cdot e_2 -e_2 \cdot e_1) +\big((\tau+id^{\otimes2})(id \otimes L(e_2)-R(e_2)\otimes id)\big) \Delta(e_1) \\
  && -\big(id \otimes L(e_1) -R(e_1)\otimes id \big)\big(\Delta(e_2)- \tau\Delta(e_2)\big) \\
  &=& \big(L(e_2) \otimes id + id \otimes L(e_2) -id \otimes R(e_2) - R(e_2)\otimes id \big) (e_1 \otimes e_1) \\
  &=& 0.
\end{eqnarray*}
Then $(A,\cdot,\Delta)$ is an $A_3$-associative bialgebra.
Therefore, the conclusion holds.
\end{proof}


\section{$A_3$-associative Yang-Baxter equation, triangular $A_3$-associative bialgebras and relative Rota-Baxter operators}

In this section, we introduce the $A_3$-associative Yang-Baxter equation, show that its skew-symmetric solution leads to triangular $A_3$-associative bialgebras, and use the relative Rota-Baxter operator to provide solutions.

\subsection{$A_3$-associative Yang-Baxter equation and triangular $A_3$-associative bialgebras.}

\begin{definition} \label{Y-Bequa}
Let $(A,\cdot)$ be an $A_3$-associative algebra and $r=\sum\limits_{i} (u_{i}\otimes v_{i}) \in A\otimes A$. Set
\begin{equation} \label{eq:asso-YB}
   AY(r)=\sum\limits_{i,j} (u_i \otimes u_j \otimes v_i \cdot v_j - u_i \otimes u_j\cdot v_i \otimes v_j + u_j \cdot u_i \otimes v_j \otimes v_i) \in A\otimes A\otimes A.
\end{equation}
The equation $AY(r)=0$ is called \textbf{$A_3$-associative Yang-Baxter equation} . If $AY(r)=0$, then $r$ is called a solution of the $A_3$-associative Yang-Baxter equation.
\end{definition}

\begin{remark}  
In \cite{Aguiar1}, the associative Yang-Baxter equation in an associative algebra is the same as \eqref{eq:asso-YB}.
\end{remark}

\begin{proposition} \label{co-co-AY}
Let $(A,\cdot)$ be an admissible $A_3$-associative algebra and $r=\sum\limits_{i} ( u_{i}\otimes v_{i}) \in A\otimes A$. Let $h:A \rightarrow End(A \otimes A)$ be a linear map defined as \eqref{eq:hx}, $\Delta_r:A \rightarrow A \otimes A$ be a linear map satisfying
\begin{equation} \label{eq:cobound}
  \Delta_r(x) = h(x)r = (id \otimes L(x) -R(x) \otimes id)r.
\end{equation}
Then \eqref{eq:asso-coalg} holds if and only if for all $x \in A$,
\begin{eqnarray}
  &\quad \quad \quad (id^{\otimes3} + \xi + \xi^2) \bigg(\big(R(x) \otimes id \otimes id -id \otimes id \otimes L(x) \big)AY(r) \notag \\
  &\quad \quad +\sum\limits_{j} \Big( \big(id \otimes R(u_j)L(x)- id \otimes L(x)R(u_j)\big) \big(r+\tau(r)\big) \otimes v_j \Big) \notag \\
  &\quad +\sum\limits_{i} \Big( u_i\otimes \big(R(x) L(v_i) \otimes id- L(v_i) R(x) \otimes id \big) \big(r+\tau(r)\big) \Big) \bigg) =0. \label{co-equi}
\end{eqnarray}
Furthermore, if $r$ is a skew-symmetric solution of the $A_3$-associative Yang-Baxter equation, then $(A,\Delta_r)$ is an $A_3$-associative coalgebra.
\end{proposition}
\begin{proof}
For all $x \in A$, we have
\begin{eqnarray*}
  && (\Delta_r \otimes id - id \otimes \Delta_r)\Delta_r(x) \\
  &=& (\Delta_r \otimes id - id \otimes \Delta_r) \big(\sum\limits_{i} (u_i \otimes x \cdot v_i - u_i \cdot x \otimes v_i) \big)  \\
  &=& \sum\limits_{i} \big( \Delta_r(u_i)\otimes x \cdot v_i - u_i\otimes \Delta_r(x \cdot v_i) -\Delta_r(u_i \cdot x) \otimes v_i +u_i \cdot x \otimes \Delta_r(v_i) \big) \\
  &=& \sum\limits_{i,j} \big( 
  \underbrace{u_j \otimes u_i \cdot v_j \otimes x\cdot v_i}_{(a1)} -\underbrace{u_j \cdot u_i \otimes v_j \otimes x\cdot v_i}_{(a2)} 
  -\underbrace{u_i \otimes u_j \otimes (x\cdot v_i) \cdot v_j}_{(a3)} +\underbrace{u_i \otimes u_j \cdot (x\cdot v_i) \otimes v_j}_{(a4)}  \\
  && -\underbrace{u_j \otimes (u_i \cdot x) \cdot v_j \otimes v_i}_{(a5)} +\underbrace{u_j\cdot(u_i \cdot x) \otimes v_j \otimes v_i}_{(a6)}
  +\underbrace{u_i \cdot x \otimes u_j \otimes v_i \cdot v_j}_{(a7)} -\underbrace{u_i \cdot x \otimes u_j\cdot v_i \otimes v_j \big)}_{(a8)}.
\end{eqnarray*}
Divide the above into three parts, and label them as follows:
\begin{eqnarray*}
  (Eq1) := \sum\limits_{i,j} \big( (a1) -(a2) +(a7) -(a8) \big),~
  (Eq2) := \sum\limits_{i,j} \big( (a4) -(a5) \big),~
  (Eq3) := \sum\limits_{i,j} \big( -(a3) +(a6) \big).
\end{eqnarray*}
Then
\begin{eqnarray*}
  (Eq1)
  &=& \sum\limits_{i,j} \big( (R(x) \otimes id \otimes id)(u_i \otimes u_j \otimes v_i \cdot v_j - u_i \otimes u_j\cdot v_i \otimes v_j) \\
  &&-(id \otimes id \otimes L(x))(u_j \cdot u_i \otimes v_j \otimes v_i - u_j \otimes u_i \cdot v_j \otimes v_i) \big), \\
  (Eq2)
  &\overset{\eqref{eq:asso-alg}}{=}& \sum\limits_{i,j} \big( u_i \otimes (x\cdot v_i) \cdot u_j \otimes v_j - u_i \otimes x \cdot (v_i\cdot u_j) \otimes v_j + u_i \otimes (v_i\cdot u_j) \cdot x \otimes v_j \\
  && -u_i \otimes v_i \cdot (u_j\cdot x) \otimes v_j \big) \\
  &=& \sum\limits_{j} \Big(\big( id \otimes R(u_j)L(x)- id \otimes L(x)R(u_j) \big) \big(r+\tau(r)\big) \otimes v_j \Big) \\
  && +\sum\limits_{i} \Big(u_i\otimes \big(R(x)L(v_i) \otimes id- L(v_i)R(x) \otimes id \big)\big(r+\tau(r)\big)\Big) \\
  && +\sum\limits_{i,j} \Big( \underbrace{v_i \otimes x\cdot (u_i \cdot u_j) \otimes v_j}_{(b1)} -\underbrace{ v_i \otimes (x\cdot u_i) \cdot u_j \otimes v_j}_{(b2)} -\underbrace{ u_i \otimes (v_i\cdot v_j) \cdot x \otimes u_j}_{(b3)}  \\
  && +\underbrace{ u_i \otimes v_i \cdot (v_j \cdot x) \otimes u_j}_{(b4)} \Big), \\  
  (Eq3)
  &\overset{\eqref{eq:admissible}}{=}& \sum\limits_{i,j} \Big( (x \cdot u_i) \cdot u_j \otimes v_j \otimes v_i - x \cdot (u_i \cdot u_j) \otimes v_j \otimes v_i + (u_j \cdot u_i) \cdot x \otimes v_j \otimes v_i \\
  &&- u_i \otimes u_j \otimes v_j \cdot (v_i \cdot x) +u_i\otimes u_j \otimes (v_j \cdot v_i) \cdot x -u_i\otimes u_j\otimes x\cdot (v_i \cdot v_j) \Big) \\
  &=& \sum\limits_{i,j} \Big( (R(x) \otimes id \otimes id)(u_j \cdot u_i \otimes v_j \otimes v_i) -(id \otimes id \otimes L(x))(u_i \otimes u_j \otimes v_i \cdot v_j)  \\
  && + \underbrace{(x \cdot u_i) \cdot u_j \otimes v_j \otimes v_i}_{(b5)} -\underbrace{ x \cdot (u_i \cdot u_j) \otimes v_j \otimes v_i}_{(b6)} -\underbrace{ u_j \otimes u_i \otimes v_i \cdot (v_j \cdot x)}_{(b7)} \\
  && +\underbrace{ u_j \otimes u_i \otimes (v_i \cdot v_j) \cdot x}_{(b8)} \Big).
\end{eqnarray*}
Label
\begin{eqnarray*}
  (Eq4) := \sum\limits_{i,j} \big( (b1) -(b2) -(b3) +(b4) + (b5) -(b6) -(b7) +(b8) \big).
 \end{eqnarray*} 
By the action of $\xi$, we have 
\begin{eqnarray*}
  && (id^{\otimes3} + \xi + \xi^2) \big(\sum\limits_{i,j} (Eq4) \big) \\
  &=& (id^{\otimes3} + \xi + \xi^2) \Big(\sum\limits_{i,j} \big( v_i \otimes x\cdot (u_i \cdot u_j) \otimes v_j -v_i \otimes (x\cdot u_i) \cdot u_j \otimes v_j -u_i \otimes (v_i\cdot v_j) \cdot x \otimes u_j \\
  && +u_i \otimes v_i \cdot (v_j \cdot x) \otimes u_j + v_i \otimes (x \cdot u_i) \cdot u_j \otimes v_j -v_i \otimes x \cdot (u_i \cdot u_j) \otimes v_j -u_i \otimes v_i \cdot (v_j \cdot x) \otimes u_j \\
  && +u_i \otimes (v_i \cdot v_j) \cdot x \otimes u_j \big) \Big) \\
  &=& 0. 
\end{eqnarray*}
Hence,
\begin{eqnarray*}
  &&(id^{\otimes3} + \xi + \xi^2)(\Delta_r \otimes id - id \otimes \Delta_r)\Delta_r(x) \\
  &=& (id^{\otimes3} + \xi + \xi^2) \bigg(\big(R(x) \otimes id \otimes id -id \otimes id \otimes L(x) \big)AY(r)  \\
  && +\sum\limits_{j} \Big( \big(id \otimes R(u_j)L(x)- id \otimes L(x)R(u_j)\big) \big(r+\tau(r)\big) \otimes v_j \Big) \\
  && +\sum\limits_{i} \Big( u_i\otimes \big(R(x) L(v_i) \otimes id- L(v_i)R(x) \otimes id \big)\big(r+\tau(r)\big) \Big) \bigg).
\end{eqnarray*}
Thus, \eqref{eq:asso-coalg} holds if and only if \eqref{co-equi} holds.
Furthermore, if $r$ is a skew-symmetric solution of the $A_3$-associative Yang-Baxter equation, then we have $r+\tau(r)=0$ and $AY(r)=0$. Hence, \eqref{co-equi} holds. Therefore, $(A,\Delta_r)$ is an $A_3$-associative coalgebra.
\end{proof}

\begin{theorem} \label{YB-bialg}
Let $(A,\cdot)$ be an admissible $A_3$-associative algebra and $r=\sum\limits_{i} u_{i}\otimes v_{i}\in A\otimes A$. Let $\Delta_r:A \rightarrow A \otimes A$ be a linear map defined as \eqref{eq:cobound}. Then \eqref{eq:asso-bialg} and \eqref{eq:asso-bialg2} hold if and only if for all $x,y \in A$,
\begin{eqnarray} 
  && \big(id \otimes L(x)L(y) -id \otimes L(x \cdot y) +L(x \cdot y) \otimes id -L(x)L(y) \otimes id +L(x) \otimes L(y) \notag  \\
  &&\quad \quad \quad + R(y) \otimes R(x) -id \otimes R(x)L(y) - R(y)L(x) \otimes id \big) \big(r+\tau(r)\big)=0,  \label{eq:bia-iff1} \\
  && \big(id\otimes R(y\cdot x) -id\otimes R(x)R(y) +id\otimes L(y)L(x) -id\otimes L(y\cdot x) -id\otimes R(y)L(x)  \notag \\
  &&\quad \quad \quad \quad +L(y) \otimes L(x) +R(x) \otimes R(y) +L(x \cdot y) \otimes id -R(x \cdot y) \otimes id  \notag \\
  &&\quad \quad \quad -L(x)L(y) \otimes id +R(y)R(x) \otimes id -L(y)R(x) \otimes id\big) \big(r+\tau(r)\big)=0. \label{eq:bia-iff2}
\end{eqnarray}
Furthermore, if $r$ is a skew-symmetric solution of the $A_3$-associative Yang-Baxter equation,
then $(A,\cdot,\Delta_r)$ is an $A_3$-associative bialgebra.
\end{theorem}
\begin{proof}
Firstly, 
for all $x,y \in A$, we have
\begin{eqnarray*}
  && (\tau-id^{\otimes2})\big(\Delta_r(x \cdot y) -(R(y)\otimes id)\Delta_r(x) -(id\otimes L(x))\Delta_r(y)\big) \\
  &=& (\tau-id^{\otimes2})\Big(\sum\limits_{i} \big( u_i \otimes (x \cdot y) \cdot v_i -u_i \cdot (x \cdot y) \otimes v_i - u_i\cdot y \otimes x\cdot v_i + (u_i \cdot x) \cdot y \otimes v_i \\
  && -u_i \otimes x\cdot (y \cdot v_i) +u_i \cdot y \otimes x \cdot v_i \big) \Big)  \\
  &=& \sum\limits_{i} \big(
  \underbrace{(x \cdot y) \cdot v_i \otimes u_i}_{(c1)} -\underbrace{u_i \otimes (x \cdot y) \cdot v_i}_{(c2)} -\underbrace{v_i \otimes u_i \cdot (x \cdot y)}_{(c3)} +\underbrace{u_i \cdot (x \cdot y) \otimes v_i}_{(c4)} -\underbrace{x\cdot v_i \otimes u_i\cdot y}_{(c5)}  \\
  && +\underbrace{u_i\cdot y \otimes x\cdot v_i}_{(c6)} +\underbrace{v_i\otimes (u_i \cdot x) \cdot y}_{(c7)} -\underbrace{(u_i \cdot x) \cdot y \otimes v_i}_{(c8)} -\underbrace{x\cdot (y \cdot v_i) \otimes u_i}_{(c9)} +\underbrace{u_i \otimes x\cdot (y \cdot v_i)}_{(c10)}  \\
  && +\underbrace{x \cdot v_i \otimes u_i \cdot y}_{(c11)} -\underbrace{u_i \cdot y \otimes x \cdot v_i}_{(c12)} \big),  \\
  && (id \otimes L(y)-R(y)\otimes id) \tau\Delta_r(x) \\
  &=& \sum\limits_{i} \big( (id \otimes L(y)-R(y)\otimes id)(x\cdot v_i \otimes u_i -v_i \otimes u_i \cdot x) \big) \\
  &=& \sum\limits_{i} \big( \underbrace{x\cdot v_i \otimes y\cdot u_i}_{(c13)} -\underbrace{(x\cdot v_i)\cdot y \otimes u_i}_{(c14)} -\underbrace{v_i \otimes y\cdot (u_i \cdot x)}_{(c15)} +\underbrace{v_i\cdot y \otimes u_i \cdot x}_{(c16)} \big), \\
  && (L(x) \otimes id -id\otimes R(x))\Delta_r(y) \\
  &=& \sum\limits_{i}\big( (L(x) \otimes id -id\otimes R(x))(u_i \otimes y\cdot v_i -u_i \cdot y\otimes v_i) \big)\\
  &=& \sum\limits_{i}\big( \underbrace{x\cdot u_i \otimes y\cdot v_i}_{(c17)} -\underbrace{u_i \otimes (y\cdot v_i)\cdot x}_{(c18)} -\underbrace{x\cdot (u_i \cdot y)\otimes v_i}_{(c19)} +\underbrace{u_i \cdot y\otimes v_i \cdot x}_{(c20)} \big).
\end{eqnarray*}
Divide the above into four parts, and label them as follows:
\begin{eqnarray*}
  (Eq5) &:=& \sum\limits_{i} \big( -(c2) -(c3) +(c7) +(c10) -(c15) -(c18) \big), \\
  (Eq6) &:=& \sum\limits_{i} \big( (c13) +(c16) +(c17) +(c20) \big), \\
  (Eq7) &:=& \sum\limits_{i} \big( (c1) +(c4) -(c8) -(c9) -(c14) -(c19) \big), \\ 
  (Eq8) &:=& \sum\limits_{i} \big( -(c5) +(c6) +(c11) -(c12) \big) =0. 
\end{eqnarray*}
Then
\begin{eqnarray*}
  (Eq5)
  &\overset{\eqref{eq:asso-alg}}{=}& \sum\limits_{i} \big(v_i \otimes x \cdot (y \cdot u_i) -v_i \otimes (x \cdot y) \cdot u_i - v_i \otimes (y \cdot u_i) \cdot x + u_i \otimes x \cdot (y \cdot v_i) -u_i \otimes (x \cdot y) \cdot v_i  \\
  && -u_i \otimes (y \cdot v_i) \cdot x \big) \\
  &=& \big(id \otimes L(x)L(y) -id \otimes L(x \cdot y) - id \otimes R(x)L(y) \big)\big(r+\tau(r)\big),\\
  (Eq6)
  &=& \big(L(x) \otimes L(y) + R(y) \otimes R(x)\big) \big(r+\tau(r)\big), \\
  (Eq7)
  &\overset{\eqref{eq:asso-alg}}{=}& \sum\limits_{i} \big( (x \cdot y) \cdot v_i \otimes u_i -x \cdot (y \cdot u_i) \otimes v_i - y \cdot (u_i \cdot x) \otimes v_i + (x \cdot y) \cdot u_i \otimes v_i +(y \cdot u_i) \cdot x \otimes v_i  \\
  && -x \cdot (y \cdot v_i) \otimes u_i - (x \cdot v_i) \cdot y \otimes u_i- x \cdot (u_i \cdot y) \otimes v_i \big) \\
  &\overset{\eqref{eq:admissible}}{=}& \sum\limits_{i} \Big( \big(L(x\cdot y) \otimes id -L(x)L(y)\otimes id\big) \big(r+\tau(r)\big) -(x \cdot u_i) \cdot y \otimes v_i +x \cdot (u_i \cdot y) \otimes v_i  \\ 
  && - (x \cdot v_i) \cdot y \otimes u_i- x \cdot (u_i \cdot y) \otimes v_i \Big) \\
  &=& \big( L(x\cdot y) \otimes id -L(x)L(y)\otimes id -R(y)L(x)\otimes id \big) \big(r+\tau(r)\big).
\end{eqnarray*}
Hence, \eqref{eq:bia-iff1} holds if and only if \eqref{eq:asso-bialg} holds.

Secondly, for all $x,y \in A$, 
\begin{eqnarray*}
  && \Delta_r(x \cdot y -y \cdot x) \\
  &=& \sum\limits_{i}\big( 
  \underbrace{u_i\otimes (x \cdot y -y \cdot x) \cdot v_i}_{(d1)} -\underbrace{u_i \cdot (x \cdot y -y \cdot x) \otimes v_i}_{(d2)} \big), \\
  && \big((\tau+id^{\otimes2})(id \otimes L(y)-R(y)\otimes id)\big) \Delta_r(x) \\
  &=& \sum\limits_{i}\big( (L(y)\otimes id +id \otimes L(y) -id\otimes R(y) -R(y)\otimes id)(u_i\otimes x \cdot v_i -u_i\cdot x \otimes v_i) \big) \\
  &=& \sum\limits_{i}\big( \underbrace{y\cdot u_i\otimes x \cdot v_i}_{(d3)} +\underbrace{u_i\otimes y\cdot(x \cdot v_i)}_{(d4)} -\underbrace{u_i\otimes (x \cdot v_i)\cdot y}_{(d5)} -\underbrace{u_i\cdot y\otimes x \cdot v_i}_{(d6)} -\underbrace{y\cdot(u_i\cdot x) \otimes v_i}_{(d7)}  \\
  && -\underbrace{u_i\cdot x \otimes y\cdot v_i}_{(d8)} +\underbrace{u_i\cdot x \otimes v_i\cdot y}_{(d9)} +\underbrace{(u_i\cdot x)\cdot y \otimes v_i}_{(d10)} \big), \\
  && (R(x)\otimes id -id \otimes L(x)) \big(\Delta_r(y)- \tau\Delta_r(y) \big) \\
  &=& \sum\limits_{i}\big( (R(x)\otimes id -id \otimes L(x))(u_i\otimes y \cdot v_i - u_i\cdot y \otimes v_i -y\cdot v_i \otimes u_i +v_i\otimes u_i\cdot y) \big) \\
  &=& \sum\limits_{i}\big( \underbrace{u_i\cdot x \otimes y \cdot v_i}_{(d11)} -\underbrace{u_i\otimes x\cdot(y \cdot v_i)}_{(d12)} -\underbrace{(u_i\cdot y)\cdot x \otimes v_i}_{(d13)} +\underbrace{u_i\cdot y \otimes x \cdot v_i}_{(d14)} -\underbrace{(y\cdot v_i)\cdot x \otimes u_i}_{(d15)} \\
  && +\underbrace{y\cdot v_i \otimes x\cdot u_i}_{(d16)} +\underbrace{v_i\cdot x\otimes u_i\cdot y}_{(d17)} -\underbrace{v_i\otimes x\cdot(u_i\cdot y)}_{(d18)}  \big).
\end{eqnarray*}
Divide the above into four parts, and label them as follows:
\begin{eqnarray*}
  (Eq9) &:=& \sum\limits_{i} \big( (d1) +(d4) -(d5) -(d12) -(d18) \big), \\
  (Eq10) &:=& \sum\limits_{i} \big( (d3) +(d9) +(d16) +(d17) \big), \\
  (Eq11) &:=& \sum\limits_{i} \big( -(d2) -(d7) +(d10) -(d13) -(d15) \big), \\
  (Eq12) &:=& \sum\limits_{i} \big( -(d6) -(d8) +(d11) +(d14) \big) =0. 
\end{eqnarray*}
Then
\begin{eqnarray*}
  (Eq9)
  &=& \sum\limits_{i} \big( u_i\otimes v_i\cdot (y \cdot x) -u_i\otimes (v_i\cdot y)\cdot x -u_i\otimes (y \cdot x)\cdot v_i +u_i\otimes y\cdot(x \cdot v_i) -u_i\otimes (x \cdot v_i)\cdot y  \\
  && +v_i\otimes u_i\cdot(y\cdot x) -v_i\otimes (y\cdot x)\cdot u_i -v_i\otimes (u_i\cdot y)\cdot x +v_i\otimes y \cdot(x\cdot u_i) -v_i\otimes (x\cdot u_i)\cdot y \big) \\
  &=& \big(id\otimes R(y \cdot x) -id\otimes L(y \cdot x) -id\otimes R(x)R(y) +id\otimes L(y)L(x) \\
  && -id\otimes R(y)L(x)\big)\big(r+\tau(r)\big), \\
  (Eq10)
  &=& \big(L(y) \otimes L(x) +R(x) \otimes R(y)\big) \big(r+\tau(r) \big), \\
  (Eq11)
  &=& \sum\limits_{i} \big( (x\cdot y) \cdot u_i \otimes v_i -x\cdot(y\cdot u_i) \otimes v_i -u_i \cdot (x \cdot y) \otimes v_i -y\cdot(u_i\cdot x) \otimes v_i +(u_i\cdot x)\cdot y \otimes v_i   \\
  && +(x\cdot y)\cdot v_i \otimes u_i -v_i\cdot(x\cdot y) \otimes u_i -x\cdot(y \cdot v_i) \otimes u_i +(v_i\cdot x)\cdot y \otimes u_i -y\cdot(v_i\cdot x) \otimes u_i \big) \\
  &=& \big(L(x \cdot y) \otimes id -R(x \cdot y) \otimes id -L(x)L(y) \otimes id +R(y)R(x) \otimes id  \\
  && -L(y)R(x) \otimes id\big) \big(r+\tau(r)\big).
\end{eqnarray*}
Hence, \eqref{eq:bia-iff2} holds if and only if \eqref{eq:asso-bialg2} holds.

Thirdly, if $r$ is a skew-symmetric solution of the $A_3$-associative Yang-Baxter equation,
then we have $r+\tau(r)=0$ and $AY(r)=0$. Hence, \eqref{eq:bia-iff1} and \eqref{eq:bia-iff2} hold. By Proposition \ref{co-co-AY}, $(A,\Delta_r)$ is an $A_3$-associative coalgebra. 
Therefore, $(A,\cdot,\Delta_r)$ is an $A_3$-associative bialgebra.
\end{proof}

\begin{definition} \label{triangular}
Let $(A,\cdot)$ be an admissible $A_3$-associative algebra, $\Delta_r:A \rightarrow A \otimes A$ be a linear map given by \eqref{eq:cobound}. The $A_3$-associative bialgebra $(A,\cdot,\Delta_r,r)$ is called \textbf{triangular} if $r \in A\otimes A$ is a skew-symmetric solution of the $A_3$-associative Yang-Baxter equation.
\end{definition}

\subsection{Relative Rota-Baxter operators of admissible $A_3$-associative algebras.}

\begin{definition} \label{R-Boperator}
Let $(A,\cdot)$ be an $A_3$-associative algebra with a representation $(l,r,V)$. A linear map $T:V\rightarrow A$ is called a \textbf{relative Rota-Baxter operator} of $(A,\cdot)$ associated to $(l,r,V)$ if 
\begin{equation} \label{eq:T-RBo}
   T(u) \cdot T(v) = T(l(Tu)v+r(Tv)u), \quad \forall~ u,v \in V.
\end{equation}
\end{definition}

Let $V$ and $A$ be vector spaces, $r \in V\otimes A$. Define a linear map $r^\sharp:V^*\rightarrow A$ as
\begin{equation} \label{eq:rsharp}
  \langle r^\sharp(u^*), a^* \rangle = \langle r, u^*\otimes a^* \rangle, \quad \forall~ u^* \in V^*, a^* \in A^*.
\end{equation}

Now we study the operator form of $r$ in a triangular $A_3$-associative bialgebra $(A,\cdot,\Delta_r)$.

\begin{proposition} \label{acircb}
Let $(A,\cdot)$ be an admissible $A_3$-associative algebra, $(R^*,L^*,A^*)$ be the coadjoint representation, and $r=\sum\limits_{i} u_{i}\otimes v_{i}\in A\otimes A$ be a skew-symmetric tensor. Let $\Delta_r:A\rightarrow A\otimes A$ be a linear map given by \eqref{eq:cobound} and $\circ:A^* \otimes A^* \rightarrow A^*$ be the linear dual of $\Delta_r$ satisfying \eqref{eq:co-mul}. Then
\begin{equation} \label{eq:acircb}
  a^* \circ b^*= R^*(r^\sharp(a^*))b^* + L^*(r^\sharp(b^*))a^*, 
\end{equation}
and 
\begin{equation} \label{eq:rshacircb}
  \langle r^\sharp(a^*)\cdot r^\sharp(b^*) -r^\sharp(a^*\circ b^*), c^* \rangle = \langle a^* \otimes b^* \otimes c^*, AY(r) \rangle, \quad \forall~a^*,b^*,c^* \in A^*.
\end{equation}
\end{proposition}
\begin{proof}
Firstly, we prove \eqref{eq:acircb}.
For all $a^*,b^*,c^* \in A^*, x\in A$, we have
\begin{eqnarray*}
  && \langle a^* \circ b^*, x \rangle
  \overset{\eqref{eq:co-mul}}{=} \langle a^* \otimes b^*, \Delta_r(x) \rangle \\
  &\overset{\eqref{eq:cobound}}{=}& \langle a^* \otimes b^*, (id \otimes L(x) -R(x) \otimes id)r \rangle \\
  &\overset{\eqref{eq:ad-dual}}{=}& \langle r, a^* \otimes L^*(x)b^*-R^*(x)a^* \otimes b^* \rangle \\
  &=& \langle r, a^* \otimes L^*(x)b^* \rangle + \langle r, b^* \otimes R^*(x)a^* \rangle \\
  &\overset{\eqref{eq:rsharp}}{=}& \langle r^\sharp(a^*), L^*(x)b^* \rangle + \langle r^\sharp(b^*), R^*(x)a^* \rangle\\
  &\overset{\eqref{eq:ad-dual}}{=}& \langle x\cdot r^\sharp(a^*), b^* \rangle +\langle r^\sharp(b^*) \cdot x, a^* \rangle \\
  &\overset{\eqref{eq:ad-dual}}{=}& \langle x, R^*(r^\sharp(a^*))b^* +L^*(r^\sharp(b^*))a^*\rangle.
\end{eqnarray*}
Hence, \eqref{eq:acircb} holds. 

Secondly, we divide the left side of \eqref{eq:rshacircb} into the following two parts.
\begin{eqnarray*}
  && \langle r^\sharp(a^*\circ b^*), c^* \rangle
  \overset{\eqref{eq:rsharp}}{=} \sum_i \langle u_i \otimes v_i, (a^* \circ b^*)\otimes c^* \rangle \\
  &=& \sum_i \langle u_i, a^* \circ b^* \rangle \langle v_i, c^* \rangle
  \overset{\eqref{eq:co-mul}}{=} \sum_i \langle \Delta_r(u_i), a^* \otimes b^* \rangle \langle v_i, c^* \rangle \\
  &\overset{\eqref{eq:cobound}}{=}& \sum_{i,j} \langle u_j\otimes u_i \cdot v_j - u_j\cdot u_i \otimes v_j, a^* \otimes b^* \rangle \langle v_i, c^* \rangle \\
  &=& \sum_{i,j} \langle a^*\otimes b^*\otimes c^*, u_j\otimes u_i \cdot v_j\otimes v_i -u_j\cdot u_i \otimes v_j\otimes v_i \rangle, \\
  && \langle r^\sharp(a^*)\cdot r^\sharp(b^*), c^* \rangle
  \overset{\eqref{eq:co-mul}}{=} \langle r^\sharp(a^*), R_\cdot^*(r^\sharp(b^*)) c^* \rangle \\
  &\overset{\eqref{eq:rsharp}}{=}& \sum_i \langle u_i \otimes v_i, a^* \otimes R_\cdot^*(r^\sharp(b^*))c^* \rangle
  =\sum_i \langle u_i, a^* \rangle \langle v_i, R_\cdot^*(r^\sharp(b^*))c^* \rangle \\
  &=& \sum_i \langle u_i, a^* \rangle \langle r^\sharp(b^*), L_\cdot^*(v_i)c^* \rangle
  \overset{\eqref{eq:rsharp}}{=} \sum_{i,j} \langle u_i, a^* \rangle \langle u_j\otimes v_j, b^* \otimes L_\cdot^*(v_i)c^* \rangle \\
  &\overset{\eqref{eq:ad-dual}}{=}& \sum_{i,j} \langle u_i, a^* \rangle \langle u_j, b^* \rangle \langle v_i\cdot v_j, c^* \rangle
  = \sum_{i,j} \langle a^*\otimes b^*\otimes c^*, u_i \otimes u_j \otimes v_i\cdot v_j \rangle.
\end{eqnarray*}
By subtracting the above equations, we have
\begin{eqnarray*}
  && \langle r^\sharp(a^*)\cdot r^\sharp(b^*)-r^\sharp(a^*\circ b^*),c^* \rangle \\
  &=& \sum_{i,j} \langle a^*\otimes b^*\otimes c^*, u_i \otimes u_j \otimes v_i\cdot v_j -u_j\otimes u_i \cdot v_j\otimes v_i +u_j\cdot u_i \otimes v_j\otimes v_i \rangle \\
  &=& \langle a^*\otimes b^*\otimes c^*, AY(r) \rangle.
\end{eqnarray*}
Hence, the conclusion holds.
\end{proof}

\begin{proposition} \label{coadsharp}
Let $(A,\cdot)$ be an admissible $A_3$-associative algebra and $r=\sum\limits_{i} (u_{i}\otimes v_{i}) \in A\otimes A$ be a skew-symmetric tensor.
\begin{enumerate}[(1)]
  \item Let $r$ be a solution of the $A_3$-associative Yang-Baxter equation. Then we have $r^\sharp:A^*\rightarrow A$ is a homomorphism of $(A,\cdot)$.
  \item Let $\Delta_r:A\rightarrow A\otimes A$ be a linear map given by \eqref{eq:cobound} and $\circ:A^* \otimes A^* \rightarrow A^*$ be the linear dual of $\Delta_r$ satisfying \eqref{eq:co-mul}. Then the following statements are equivalent:
  \begin{enumerate}[(i)]
  \item r is a solution of the $A_3$-associative Yang-Baxter equation such that $(A,\cdot,\Delta_r)$ is a triangular $A_3$-associative bialgebra.
  \item $r^\sharp:A^*\rightarrow A$ is a relative Rota-Baxter operator of $(A,\cdot)$ associated to $(R^*,L^*,A^*)$, that is,
      \begin{equation} \label{eq:RBope}
        r^\sharp(a^*)\cdot r^\sharp(b^*) = r^\sharp(R^*(r^\sharp(a^*))b^* +L^*(r^\sharp(b^*))a^*), \quad \forall~ a^*,b^* \in A^*.
      \end{equation}
  \end{enumerate}
\end{enumerate}
\end{proposition}
\begin{proof}
(1) By $AY(r)=0$ and \eqref{eq:rshacircb}, we have
\begin{equation}
  r^\sharp(a^*\circ b^*) = r^\sharp(a^*)\cdot r^\sharp(b^*), \quad \forall~ a^*,b^* \in A^*.
\end{equation}
Therefore, $r^\sharp:A^*\rightarrow A$ is a homomorphism of $(A,\cdot)$.

(2) It follows directly from Proposition \ref{acircb}.
\end{proof}

A Connes cocycle ({\rm\cite{BaiC2}}) on an associative algebra $(A,\cdot)$ is a skew-symmetric bilinear form $\omega(\cdot,\cdot):A \times A \rightarrow \mathbb{K}$ satisfying
\begin{equation} \label{Con-co}
  \omega(x\cdot y,z)+ \omega(y\cdot z,x) + \omega(z\cdot x,y)=0, \quad \forall~ x,y,z \in A.
\end{equation}

\begin{theorem} \label{omegasharp}
Let $(A,\cdot)$ be an admissible $A_3$-associative algebra, $r\in A\otimes A$ be a skew-symmetric tensor. If $r^\sharp:A^*\rightarrow A$ is invertible and $\omega$ is a skew-symmetric bilinear form on $(A,\cdot)$ defined by
\begin{equation} \label{eq:omega}
  \omega(x,y)=\langle (r^\sharp)^{-1}(x),y \rangle, \quad \forall~ x,y \in A,
\end{equation}
then $r$ is a solution of the $A_3$-associative Yang-Baxter equation if and only if $(r^\sharp)^{-1}$ is a Connes cocycle.
\end{theorem}
\begin{proof}

Let $a^*,b^*,c^*\in A^*$ and $x=r^\sharp(a^*),y=r^\sharp(b^*),z=r^\sharp(c^*)$. Then we have
\begin{eqnarray*}
  && \langle (r^\sharp)^{-1}(x\cdot y),z \rangle + \langle (r^\sharp)^{-1}(y\cdot z),x \rangle + \langle (r^\sharp)^{-1}(z\cdot x),y \rangle  \\
  &\overset{\eqref{eq:omega}}{=}& \omega(x\cdot y,z) +\omega(y\cdot z,x) +  \omega(z\cdot x,y)  \\
  &\overset{\eqref{eq:omega}}{=}& -\langle (r^\sharp)^{-1}(z), x\cdot y -\langle a^*, y\cdot z \rangle -\langle b^*, z\cdot x \rangle \rangle \\
  &\overset{\eqref{eq:ad-dual}}{=}& - \langle (r^\sharp)^{-1}(z), x\cdot y \rangle -\langle L^*(y)a^* + R^*(x)b^*, z \rangle  \\
  &=& \langle x\cdot y, (r^\sharp)^{-1}(z) \rangle -\langle r^\sharp \big(L^*(y)a^* + R^*(x)b^* \big), (r^\sharp)^{-1}(z) \rangle \\
  &=& \langle r^\sharp(a^*)\cdot r^\sharp(b^*)- r^\sharp \big(L^*(r^\sharp(b^*))a^* + R^*(r^\sharp(a^*))b^* \big), c^* \rangle. 
\end{eqnarray*}
Hence, $(r^\sharp)^{-1}$ is a Connes cocycle if and only if $r^\sharp$ is a relative Rota-Baxter operator of $(A,\cdot)$ associated to $(R^*,L^*,A^*)$.
Then by Proposition \ref{coadsharp} (2), the conclusion holds.
\end{proof}

\begin{corollary}
Let $(A,\cdot)$ be an admissible $A_3$-associative algebra and $r\in A\otimes A$ be a skew-symmetric solution of the $A_3$-associative Yang-Baxter equation in $(A,\cdot)$. If the admissible $A_3$-associative algebra structure $\circ$ on $A^*$ is given by \eqref{eq:acircb},  then $r^\sharp:A^*\rightarrow A$ is a homomorphism of $(A,\cdot)$. 
\end{corollary}
\begin{proof}
For all $a^*,b^* \in A^*, x\in A$, we have
\begin{eqnarray*}
  && \langle a^* \circ b^*, x \rangle
  = \langle R^*(r^\sharp(a^*))b^*+L^*(r^\sharp(b^*))a^*, x \rangle \\
  &\overset{\eqref{eq:ad-dual}}{=}& \langle b^*, x \cdot r^\sharp(a^*) \rangle +  \langle a^*, r^\sharp(b^*) \cdot x \rangle \\
  &\overset{\eqref{eq:omega}}{=}& \omega (r^\sharp(b^*), x \cdot r^\sharp(a^*)) + \omega(r^\sharp(a^*), r^\sharp(b^*) \cdot x) \\
  &\overset{Theorem~ \ref{omegasharp}}{=}& -\omega (x, r^\sharp(a^*) \cdot r^\sharp(b^*)) \\
  &\overset{\eqref{eq:omega}}{=}& \langle (r^\sharp)^{-1}\big(r^\sharp(a^*)\cdot r^\sharp(b^*)\big),x \rangle.
\end{eqnarray*}
Hence, we have
\begin{equation} \label{rsisomor}
  a^* \circ b^* = (r^\sharp)^{-1} (r^\sharp(a^*)\cdot r^\sharp(b^*)), \quad \forall~ a^*,b^* \in A^*.
\end{equation} 
Therefore, $r^\sharp:A^*\rightarrow A$ is a homomorphism of $(A,\cdot)$.
\end{proof}

Let $(A,\cdot)$ be an admissible $A_3$-associative algebra with an adjoint representation $(L_\cdot,R_\cdot,A)$, $A^*$ be the dual space of $A$. Let $T:A\rightarrow A$ be a linear map, $T^*:A^* \rightarrow A^*$ be the dual map of $T$, defined by
\begin{equation} \label{eq:TT}
  \langle T^*(a^*),x \rangle = \langle a^*, T(x) \rangle, \quad \forall~ a^*\in A^*, x\in A.
\end{equation}

\begin{lemma} \label{TT:rRB}
Let $(A,\cdot)$ be an admissible $A_3$-associative algebra with an adjoint representation $(L_\cdot,R_\cdot,A)$. Let $\widetilde{T}: (A\oplus A^*)^* \rightarrow A\oplus A^*$ be a linear map defined as $\widetilde{T}(a^*+x)= T^*(a^*) -T(x)$. 
Let $D=A\oplus A^*$ and $\cdot_D: D\otimes D \rightarrow D$ be a linear map defined as 
\begin{equation} \label{eq:semi-TTm}
  (x+a^*)\cdot_D(y+b^*) = x\cdot y + R_\cdot^*(x)b^* + L_\cdot^*(y)a^*, \quad \forall~ x,y\in A, a^*,b^*\in A^*.
\end{equation}
Then $T:A\rightarrow A$ is a relative Rota-Baxter operator of $(A,\cdot)$ associated to $(L_\cdot,R_\cdot,A)$ if and only if $\widetilde{T}: D^* \rightarrow D$ is a relative Rota-Baxter operator of $(D,\cdot_D)$ associated to $(R_D^*,L_D^*,D^*)$.
\end{lemma}
\begin{proof}
By Proposition \ref{adm-semi-dp}, $(D,\cdot_D)$ is an admissible $A_3$-associative algebra, we denote the adjoint representation of $(D,\cdot_D)$ by $(R_D^*,L_D^*,D^*)$.
Next, we prove that 
$\widetilde{T}: D^* \rightarrow D$ is a relative Rota-Baxter operator of $(D,\cdot_D)$ associated to $(R_D^*,L_D^*,D^*)$.
For all $a^*, b^*,c^* \in A^*, x,y\in A$, we have
\begin{eqnarray*}
  \langle R_D^*\big(T(x)\big)y, c^* \rangle = \langle y, c^* \cdot_D T(x) \rangle 
  \overset{\eqref{eq:semi-TTm}}{=} \langle y, L_\cdot^*\big(T(x)\big)c^* \rangle 
  = \langle L_\cdot\big(T(x)\big)y, c^* \rangle. 
\end{eqnarray*}
Then $R_D^*\big(T(x)\big)y=L_\cdot\big(T(x)\big)y$.
Similarly, $L_D^*\big(T(y)\big)x=R_\cdot\big(T(y)\big)x$. Then we have
\begin{eqnarray*}
  && \widetilde{T}(x+a^*) \cdot_D \widetilde{T}(y+b^*) - \widetilde{T} \Big(R_D^*\big(\widetilde{T}(x+a^*)\big)(y+b^*) +L_D^*\big(\widetilde{T} (y+b^*)\big)(x+a^*)\Big) \\
  &=& \big(\widetilde{T}(x)+\widetilde{T}(a^*)\big) \cdot_D \big(\widetilde{T}(y)+\widetilde{T}(b^*)\big) -\widetilde{T} \Big(R_D^* \big(-T(x)+T^*(a^*)\big)(y+b^*) \\
  && +L_D^*\big(-T(y)+T^*(b^*)\big)(x+a^*)\Big) \\
  &\overset{\eqref{eq:semi-TTm}}{=}& \big(-T(x)\big)\cdot \big(-T(y)\big) + R_\cdot^*\big(-T(x)\big)\big(T^*(b^*)\big) +L_\cdot^* \big(-T(y)\big)\big(T^*(a^*)\big) \\
  && -\widetilde{T}\Big(-R_D^*\big(T(x)\big)v - R_D^*\big(T(x)\big)b^* + R_D^*\big(T^*(a^*)\big)y -L_D^*\big(T(y)\big)x -L_D^*\big(T(y)\big)a^* \\
  && +L_D^*\big(T^*(b^*)\big)x\Big) \\
  &=& \underbrace{T(x) \cdot T(y)}_{(e1)} -\underbrace{R_\cdot^*\big(T(x)\big)\big(T^*(b^*)\big)}_{(e2)} -\underbrace{ L_\cdot^*\big(T(y)\big)\big(T^*(a^*)\big)}_{(e3)} -\underbrace{T\big(L_\cdot(T(x))y\big)}_{(e4)} \\
  && +\underbrace{T^*\Big(R_D^*\big(T(x)\big)b^*\Big)}_{(e5)} -\underbrace{T^*\Big(R_D^*\big(T^*(a^*)\big)y\Big)}_{(e6)} -\underbrace{ T\Big(R_\cdot\big(T(y)\big)x\Big)}_{(e7)} +\underbrace{T^*\Big(L_D^*\big(T(y)\big)a^*\Big)}_{(e8)} \\
  && -\underbrace{T^*\Big(L_D^*\big(T^*(b^*)\big)x\Big)}_{(e9)}.
\end{eqnarray*}
Take the three items from the above respectively and let
\begin{eqnarray*}
  (Eq13) := -(e2) +(e5) -(e9),~ (Eq14) := -(e3) -(e6) +(e8),~ (Eq15) := (e1) -(e4) -(e7).
\end{eqnarray*}
For all $z\in A$, we have
\begin{eqnarray*}
  &&\langle (Eq13), z \rangle =
  \langle - R_\cdot^*\big(T(x)\big)\big(T^*(b^*)\big) +T^*\Big(R_D^*\big(T(x)\big)b^*\Big) - T^*\Big(L_D^*\big(T^*(b^*)\big)x\Big), z \rangle \\
  &\overset{\eqref{eq:TT}}{=}& -\langle T^*(b^*), R_\cdot\big(T(x)\big)z \rangle + \langle R_D^*\big(T(x)\big)b^*, T(z) \rangle - \langle L_D^*\big(T^*(b^*)\big)x, T(z) \rangle \\
  &\overset{\eqref{eq:TT}}{=}& -\langle b^*, T\Big(R_\cdot\big(T(x)\big)z\Big) \rangle + \langle b^*, T(z)\cdot T(x) \rangle - \langle x, T^*(b^*) \cdot_D T(z) \rangle \\
  &\overset{\eqref{eq:semi-TTm}}{=}& -\langle b^*, T\Big(R_\cdot\big(T(x)\big)z\Big) \rangle + \langle b^*, T(z)\cdot T(x) \rangle - \langle x, L_\cdot^*\big(T(z)\big) T^*(b^*) \rangle \\
  &=& -\langle b^*, T\Big(R_\cdot\big(T(x)\big)z\Big) \rangle + \langle b^*, T(z)\cdot T(x) \rangle - \langle T\Big(L_\cdot\big(T(z)\big)x\Big), b^* \rangle.
\end{eqnarray*}
Similarly, we have
\begin{eqnarray*}
  \langle (Eq14), z \rangle 
  = -\langle a^*, T\Big(L_\cdot\big(T(y)\big)z\Big) \rangle + \langle a^*, T(y)\cdot T(z) \rangle - \langle T\Big(R_\cdot\big(T(z)\big)y\Big), a^* \rangle.
\end{eqnarray*}
Because $T$ is a relative Rota-Baxter operator of $(A,\cdot)$ associated to $(L_\cdot,R_\cdot,A)$, then 
$$(Eq13) =(Eq14) =(Eq15) =0.$$
Hence, $\widetilde{T}: D^* \rightarrow D$ is a relative Rota-Baxter operator of $(D,\cdot_D)$ associated to $(R_D^*,L_D^*,D^*)$.
Conversely, if $\widetilde{T}: D^* \rightarrow D$ is a relative Rota-Baxter operator of $(D,\cdot_D)$ associated to $(R_D^*,L_D^*,D^*)$, then 
$$(Eq13) =(Eq14) =(Eq15) =0,$$
that is, $T$ is a relative Rota-Baxter operator of $(A,\cdot)$ associated to $(L_\cdot,R_\cdot,A)$.
Thus, the conclusion holds.
\end{proof}

\begin{theorem}
Let $(A,\cdot)$ be an admissible $A_3$-associative algebra with an adjoint representation $(L_\cdot,R_\cdot,A)$. Let $\widetilde{T}: (A\oplus A^*)^* \rightarrow A\oplus A^*$ be a linear map defined as $\widetilde{T}(a^*+x)= T^*(a^*) -T(x)$, $T:A\rightarrow A$ be a linear map, which is identified as an element in $ A\otimes A^* \hookrightarrow (A\ltimes_{R_\cdot^*,L_\cdot^*} A^*) \otimes (A\ltimes_{R_\cdot^*,L_\cdot^*} A^*)$, given by
\begin{equation}
  \langle T, (a^*+x)\otimes(b^*+y) \rangle = \langle \widetilde{T}(a^*+x),b^*+y \rangle, \quad \forall~ x,y\in A, a^*,b^*\in A^*.
\end{equation}
Let $D=A\oplus A^*$ and $\cdot_D: D\otimes D \rightarrow D$ be a linear map defined as \eqref{eq:semi-TTm}.
Then $T:A\rightarrow A$ is a relative Rota-Baxter operator of $(A,\cdot)$ associated to $(L_\cdot,R_\cdot,A)$ if and only if $r=T -\tau(T)$ is a skew-symmetric solution of the $A_3$-associative Yang-Baxter equation in $A\ltimes_{R_\cdot^*,L_\cdot^*} A^*$.
\end{theorem}
\begin{proof}
By Lemma \ref{TT:rRB}, $T:A\rightarrow A$ is a relative Rota-Baxter operator of $(A,\cdot)$ associated to $(L_\cdot,R_\cdot,A)$ if and only if
$\widetilde{T}: D^* \rightarrow D$ is a relative Rota-Baxter operator of $(D,\cdot_D)$ associated to $(R_D^*,L_D^*,D^*)$.

Identify a linear map $T:A \rightarrow A$ as an element in $D\otimes D$ through the injective map
$$Hom_\mathbb{K} (A,A) \cong A\otimes A^* \hookrightarrow D\otimes D.$$
By taking $r^\sharp =\widetilde{T} \in D\otimes D$ and Proposition \ref{coadsharp} (2), it follows that $r=T -\tau(T)$ is a solution of the $A_3$-associative Yang-Baxter equation in $(D,\cdot_D)$ if and only if $\widetilde{T}: (A\oplus A^*)^* \rightarrow A\oplus A^*$ is a relative Rota-Baxter operator of $(D,\cdot_D)$ associated to $(R_D^*,L_D^*,D^*)$.
Therefore, the conclusion holds.
\end{proof}

{\bf Acknowledgements:}
 This work is supported by the Natural Science Foundation of Zhejiang Province (Grant No. LZ25A010002). The second author would like to thank Guilai Liu for
helpful discussions on  $A_3$-associative algebras.

\end{document}